\documentclass[a4paper,11pt,leqno]{article}
\pdfoutput=1
\usepackage{amsmath,amsfonts,amsthm,amssymb}
\usepackage[T2A,T1]{fontenc}
\usepackage[utf8]{inputenc}
\usepackage{CJKutf8} 
\usepackage[russian,greek,english]{babel}
\usepackage[left=3cm, right=3cm, top=3cm, bottom=3.5cm]{geometry}
\usepackage{graphicx}
\usepackage{enumerate}

\usepackage[hyphens]{url}
\usepackage{hyperref}

\hypersetup{
 pdftitle={Hamilton combinatorics},
 pdfauthor={Bernold Fiedler},
 colorlinks=true,
 linkcolor=blue,
 citecolor=blue,
 filecolor=blue,
 urlcolor=blue}

 \usepackage{afterpage}

\usepackage{wrapfig} 
\usepackage[format=plain,labelfont=bf,font=small]{caption}
\usepackage{xcolor}
\usepackage[arrow, matrix, curve]{xy}
\usepackage{float}

\usepackage{caption}
\usepackage{tabulary}
\usepackage{array}
\newcolumntype{N}[1]{>{\centering\arraybackslash}m{#1}}

\newcommand{\transv}{\mathrel{\text{\tpitchfork}}}
\makeatletter
\newcommand{\tpitchfork}{%
  \vbox{
    \baselineskip\z@skip
    \lineskip-.52ex
    \lineskiplimit\maxdimen
    \m@th
    \ialign{##\crcr\hidewidth\smash{$-$}\hidewidth\crcr$\pitchfork$\crcr}
  }%
}
\makeatother
\usepackage{latexsym}

\usepackage[notref,notcite,color, final 
]{showkeys}

\definecolor{refkey}{rgb}{1,0,0}
\definecolor{labelkey}{rgb}{1,0,0}

  \mathchardef\ordinarycolon\mathcode`\:
  \mathcode`\:=\string"8000
  \begingroup \catcode`\:=\active
    \gdef:{\mathrel{\mathop\ordinarycolon}}
  \endgroup
  
\theoremstyle{plain}
\newtheorem{theorem}{Theorem}[section]
\newtheorem{lemma}[theorem]{Lemma}
\newtheorem{proposition}[theorem]{Proposition}
\newtheorem{corollary}[theorem]{Corollary}

\newtheorem{definition}[theorem]{Definition}

\newcommand\eps{\varepsilon}

\renewcommand\rho{\varrho}
\renewcommand\phi{\varphi}

\newcommand{\bN}{\mathbb{N}}
\newcommand{\bZ}{\mathbb{Z}}

\newcommand{\bR}{\mathbb{R}}
\newcommand{\bC}{\mathbb{C}}
\newcommand{\bS}{\mathbb{S}}
\newcommand{\bP}{\mathbb{P}}
\newcommand{\bO}{\mathbb{O}(2)}
\newcommand{\bSO}{\mathbb{SO}(2)}

\def\cA{\mathcal{A}}
\def\cE{\mathcal{E}}
\def\cR{\mathcal{R}}
\def\cP{\mathcal{P}}
\def\cF{\mathcal{F}}
\def\cH{\mathcal{H}}
\def\cC{\mathcal{C}}

\def\cS{\mathcal{S}}

\def\cM{\mathcal{M}}
\def\cN{\mathcal{N}}
\def\cO{\mathcal{O}}

\def\fS{\mathfrak{S}}
\def\fa{\mathfrak{a}}
\def\fc{\mathfrak{c}}

\def\cFP{\mathcal{F}_\mathbf{g}^\mathcal{P}}
\def\cRP{\mathcal{R}_\mathbf{g}^\mathcal{P}}

\def\cCP{\mathcal{C}_\mathbf{g}^\mathcal{P}}

\def\cAgP{\mathcal{A}_\mathbf{g}^\mathcal{P}}
\def\cEgP{\mathcal{E}_\mathbf{g}^\mathcal{N}}

\def\cHgP{\mathcal{H}_\mathbf{g}^\mathcal{P}}
\def\cCgP{\mathcal{C}_\mathbf{g}^\mathcal{P}}

\def\cAgN{\mathcal{A}_\mathbf{g}^\mathcal{N}}
\def\cEgN{\mathcal{E}_\mathbf{g}^\mathcal{N}}

\def\cHgN{\mathcal{H}_\mathbf{g}^\mathcal{N}}
\def\cCgN{\mathcal{C}_\mathbf{g}^\mathcal{N}}

\def\SSOP{\mathrm{Sturm}^\cP(\mathbf{u},\mathbf{u}_\mathbf{x})}
\def\SOP{\mathrm{Sturm}^\cP(\mathbf{u},\mathbf{u}_\mathbf{x}^2)}
\def\SHP{\mathrm{Sturm}^\cP(\mathbf{u})}
\def\SxP{\mathrm{Sturm}^\cP(\mathbf{x},\mathbf{u},\mathbf{u}_\mathbf{x})}
\def\SSON{\mathrm{Sturm}^\cN(\mathbf{u},\mathbf{u}_\mathbf{x})}
\def\SON{\mathrm{Sturm}^\cN(\mathbf{u},\mathbf{u}_\mathbf{x}^2)}
\def\SHN{\mathrm{Sturm}^\cN(\mathbf{u})}
\def\SxN{\mathrm{Sturm}^\cN(\mathbf{x},\mathbf{u},\mathbf{u}_\mathbf{x})}
\def\SO{\mathrm{Sturm}(\mathbf{u},\mathbf{u}_\mathbf{x}^2)}
\def\SH{\mathrm{Sturm}(\mathbf{u})}

\begin{document}

\title{\LARGE{\textbf{
Counting Hamiltonian Sturm permutations:\\
generating functions and Gaussian distributions\\}}}

\author{
\\[1em]
\Large{\emph{in memoriam}}
\\ 
\\[1em]
Bernold Fiedler* and Carlos Rocha**†
\\[2cm]
}

\date{\small{version of \today}}
\maketitle
\thispagestyle{empty}
\vfill
\setlength{\parindent}{0cm}
*\\
        Institut f\"ur Mathematik\\
        Freie Universit\"at Berlin\\
        Arnimallee 7, D--14195 Berlin, GERMANY\\
        \url{fiedler@math.fu-berlin.de}\\
        \url{http://dynamics.mi.fu-berlin.de}\\[5mm]
**\\
        Centro de An\'alise Matem\'atica, Geometria e Sistemas Din\^amicos,\\
				Instituto Superior T\'ecnico, Universidade de Lisboa\\
        Av. Rovisco Pais 1, 1049--001 Lisbon, PORTUGAL\\


\newpage
\pagestyle{plain}
\pagenumbering{roman}
\setcounter{page}{1}

\begin{abstract}
\noindent
Our combinatorial analysis is motivated by the global dynamics of PDE scalar reaction-diffusion semiflows 
\begin{equation}
\label{aPDEN}
\mathbf{u_t} = \mathbf{u_{xx}} + \mathbf{g}(\mathbf{u}),\qquad 0<\mathbf{x}<1,
\end{equation}
for the solutions $\mathbf{u}=\mathbf{u}(\mathbf{t},\mathbf{x}),\ \mathbf{t}\geq 0$, under Neumann boundary conditions $\mathbf{u_x}=0$ at $\mathbf{x}=0,1$.
For dissipative $C^2$ nonlinearities $\mathbf{g}$, the PDE \eqref{aPDEN} possesses a compact global attractor $\cA=\cA_\mathbf{g}$.
Dissipativity requires $\mathbf{g}(\mathbf{u})\cdot\mathbf{u}<0$ for large $|\mathbf{u}|$.
Under the generic nondegeneracy assumption that all equilibrium solutions $\mathbf{u}(\mathbf{t},\mathbf{x})=\mathbf{v}(\mathbf{x})$ are hyperbolic, the number $2n+1$ of all equilibria $\mathbf{v}$ is finite and odd.
The global attractors $\cA=\cA_\mathbf{g}$ can then be classified by the orderings of their equilibria $\mathbf{v}$ at the boundaries $\mathbf{x}=0,1$.
We encode the boundary orders as \emph{Hamiltonian Sturm permutations}.
The name ``Sturm'' refers to nodal properties of PDE solutions $\mathbf{u}(\mathbf{t},\mathbf{x})$.
``Hamiltonian'' refers to the second order pendulum ODE for equilibria $\mathbf{v}(\mathbf{x})$:
\begin{equation}
\label{aODE}
0 = \mathbf{v_{xx}} + \mathbf{g}(\mathbf{v})\,.
\end{equation}
We determine the generating function $a(z)=\sum_n a_nz^n$ for the counts $a_n$ of Hamiltonian Sturm permutations. 
For $n\rightarrow\infty$, this provides explicit asymptotics of $a_n$\,.\\[0.3em]
\noindent
The prerequisite analysis studies the periodic orbits $(\mathbf{v},\mathbf{v_x})$ of ODE \eqref{aODE}
with period 2 and minimal periods $T=2/\ell>0$, for integer $\ell$. 
Let $q$ count those periodic orbits, and let $2r+1$ denote the odd number of (simple) zeros of $\mathbf{g}$.
Since each periodic orbit of \eqref{aODE} generates two Neumann equilibria of \eqref{aPDEN}, the number $2n+1$ of PDE equilibria satisfies $n=r+q$.
Summing over those $r,q$, we obtain $a_n=\sum b_{rq}$\,, where $b_{rq}$ is the refined count of Hamiltonian Sturm permutations with $2r+1$ spatially homogeneous equilibria and $2q$ spatially non-homogeneous equilibria.
We also determine the explicit generating function $b(x,y)=\sum_{r,q} b_{rq}\,x^ry^q$\,.\\[0.3em]
\noindent
This implies asymptotically Gaussian distributions of the probabilities $p_{nr}=b_{rq}/a_n$ with $r+q=n$, asymptotically for large $n$.
We provide leading terms for means $\boldsymbol{\mu}_n$ and variances $\boldsymbol{\sigma}_n^2$\,, with error estimates of order $1/n$.
All asymptotics are based on work by Flajolet and Sedgewick.\\[0.3em]
\noindent
After some numerical illustrations, we conclude with remarks on nonlinearities $\mathbf{g}(\mathbf{u},\mathbf{u_x})$ under periodic boundary conditions $\mathbf{x}\in\bS^1=\bR/2\bZ$.
This gives rise to $q$ rotating or frozen PDE waves, instead of $2q$ spatially non-homogeneous Neumann equilibria.
\end{abstract}

\newpage
{\small{
\tableofcontents
\listoffigures
\listoftables
}}


\newpage
\pagenumbering{arabic}
\setcounter{section}{-1}
\setcounter{page}{1}
\setcounter{equation}{0}
\vspace{1cm}

\section{In memoriam Carlos Rocha}
\numberwithin{equation}{section}
\numberwithin{figure}{section}
\numberwithin{table}{section}

It was forty years ago, in July 1986 at sunny Lisboa, when Jack Hale introduced us to each other.
The occasion was the conference on \emph{``Dynamics of Infinite Dimensional Systems''} there \cite{Lisboa86}.
The first of our 21 joint publications which followed, only appeared ten years later \cite{firo96}.
It was based on fundamental progress in the global analysis of parabolic attractors via the concept of \emph{Sturm permutations}, which had been introduced by Giorgio Fusco and Carlos Rocha \cite{furo91}, meanwhile.

In the present paper, we recall and survey a small part of these developments.
Special focus is on the role of ODE \emph{time maps}, or \emph{period maps}, in the Hamiltonian setting \eqref{aPDEN}, \eqref{aODE} of the abstract.
Carlos was a veritable virtuoso in their use \cite{firo99,roc07,czro08,roc12,firowo12a,rofi24,roc25,rofiln}.
Where I had to rely on programming, he would look at time maps, smile, say ``djub-djub'' and, with a snap of his fingers, draw the associated Sturm permutations and just ``see'' the \emph{connection graphs} of their associated PDE \emph{global attractors}.
Of course, all these concepts will properly be introduced, below.

For decades, our work and our many mutual visits have been supported by FCT/Portugal and the Deutsche Forschungsgemeinschaft (DFG), under various projects and grant numbers, and by our universities.

It was after our previous survey \cite{rofiln} that Carlos brought up the question of counting the resulting Sturm permutations, in a combinatorial way.
Although this question remains very challenging, for general nonlinearities $\mathbf{g}=\mathbf{g}(\mathbf{x},\mathbf{u},\mathbf{u}_\mathbf{x})$, we were making happy progress in the \emph{Hamiltonian} case $\mathbf{g}=\mathbf{g}(\mathbf{u})$ -- thanks to Carlos style ``djubology'' on period maps.

Happy progress we made, indeed, until it was disrupted by ``the silencer of all happiness and sunderer of all bonds of friendship''.
Left behind, I try to patch things up, as best as I can, in our present joint work.

I mourn the irreplaceable loss of much more than a colleague and coauthor.
Carlos has become a most precious friend.
He opened his family and circle of friends to me with unparalleled hospitality.
I will forever admire his calm and conciliatory personality, always ready -- and going out of his way -- to help anyone in need.
Welch wunderbarer Mensch.

May he live on, in any of us who follow him.

\bigskip
Berlin, October 2026 \hfill Bernold Fiedler

\section{Introduction and outline}\label{Intro}

The simplest semilinear scalar reaction-diffusion equation takes the form
\begin{equation}
\label{gu}
\mathbf{u_t} = \mathbf{u_{xx}} + \mathbf{g}(\mathbf{u})\,.
\end{equation}
Similar models arise, e.g., in engineering, physics, chemistry, biology, sociology, etc. 
See for example \cite{Galak,hen81}.
Let us first consider solutions $\mathbf{u}=\mathbf{u}(\mathbf{t},\mathbf{x})$ on the spatial domain $0<\mathbf{x}<1$, under ``no-flux'' \emph{Neumann boundary conditions} $\cN$\,:
\begin{equation}
\label{Nbc}
\mathbf{u_x}=0,\quad \textrm{ at } \mathbf{x}=0,1\,.
\end{equation}
We assume the nonlinearity $\mathbf{g}{:}\ \bR\rightarrow \bR$ to be $C^2$-smooth, for simplicity. 
Then the initial value problem \eqref{gu} for times $\mathbf{t}\geq0$ generates a local semiflow of solutions, e.g. on suitable Sobolev spaces 
$\mathbf{u}(\mathbf{t},\cdot)\in \mathbf{X}\hookrightarrow C^1$.
See for example \cite{hen81,paz83}, and further comments in section \ref{Back}.
In section \ref{Per} we will also consider the analogous setting for \emph{periodic boundary conditions} $\cP$\,:
\begin{equation}
\label{pbc}
\mathbf{x}\in\bS^1=\bR/2\bZ\,.
\end{equation}

In the following, we always assume the nonlinearity $\mathbf{g}$ to be \emph{dissipative}, i.e.\ a certain large ball in $\mathbf{X}$ eventually attracts all solutions \cite{hal88}.
For PDE \eqref{gu} this amounts to a sign condition
\begin{equation}
\label{disign}
\mathbf{g}(\mathbf{u})\cdot\mathbf{u}<0\,,
\end{equation}
for sufficiently large $|\mathbf{u}|\!\in\!\bR$; see \cite{mana97}.
In view of spatially homogeneous equilibria, the sign condition is also necessary for dissipativity.

In particular, dissipative semiflows are global. 
Ladyzhenskaya \cite{lady} has observed how, moreover, all solutions $\mathbf{u}(\mathbf{t},\cdot)\!\in\! \mathbf{X}$ are attracted to a maximal compact, forward and backward time-invariant subset $\cA\!\subset\!\mathbf{X}$, which is called the \emph{global attractor}.
See also \cite{bavi92, hamaol02, raugel, robinson}.
In sections \ref{Back} and \ref{Sign}, we address global attractors $\cA$ for PDEs \eqref{gu},  \eqref{gxuux}, more specifically.
See figure \ref{figCI3}(e) for a first example.

The PDE \eqref{gu} possesses a \emph{gradient-like structure} with respect to the explicit \emph{Lyapunov functional}
\begin{equation}
\label{lyap}
\begin{aligned}
    V(\mathbf{u}(\mathbf{t},\cdot))&:= \int_0^1 \left(\tfrac{1}{2}\mathbf{u}_\mathbf{x}^2-\mathbf{G}(\mathbf{u}) \right)d\mathbf{x} \,;   \\
    \tfrac{d}{d\mathbf{t}}V(\mathbf{u}(\mathbf{t},\cdot))&\phantom{:}=- \int_0^1 \ \mathbf{u}_\mathbf{t}^2\;d\mathbf{x}\ \leq\ 0\,;
\end{aligned}
\end{equation}
see \cite{hen81}.
Here $\mathbf{G}'\!=\!\mathbf{g}$ defines the primitive function, or \emph{potential}, $\mathbf{G}$ of $\mathbf{g}$.
More specifically, the global attractor $\cA$ consists of the set $\cE$ of time-independent equilibria $\mathbf{u}(\mathbf{t},\mathbf{x})\!=\!\mathbf{v}(\mathbf{x})$, and the set $\cH$ of \emph{heteroclinic orbits} $\mathbf{u}(\mathbf{t},\cdot)\rightarrow\mathbf{v}_\pm \textrm{ for } \mathbf{t}\rightarrow\pm\infty$ between distinct equilibria $\mathbf{v}_\pm\in\cE$.
We denote such heteroclinic connections as
\begin{equation}
\label{het}
\mathbf{u}\!:\ \mathbf{v}_-\leadsto \mathbf{v}_+\,.
\end{equation}

Let us also assume \emph{spectral nondegeneracy} of $\mathbf{g}$, i.e.\ we assume all equilibria $\mathbf{v}$ to be \emph{hyperbolic}.
This means that $\lambda\!=\!0$ is not an eigenvalue of the Sturm-Liouville eigenvalue problem which arises by linearization of PDE \eqref{gu} at any equilibrium $\mathbf{v}$.
Emphasizing the gradient-like structure, the finite number of strictly positive eigenvalues is also called the \emph{Morse index} $i(\mathbf{v})$ of $\mathbf{v}$.
We can then represent the global attractor $\cA$ by its (finite, directed, loop-free) connection graph $\cC$, as follows. 
Vertices of $\cC$ are the equilibria.
Directed edges $\mathbf{v}_-\leadsto \mathbf{v}_+$ indicate the existence of a heteroclinic orbit \eqref{het}.
Of course, we lose information on the precise topology of the attractor decomposition
\begin{equation}
\label{AEH}
\cA=\cE\cup\cH
\end{equation}
into equilibria and heteroclinic orbits, when passing to the connection graph $\cC$.
See figure \ref{figCI3}(f), left, for a first example, based on the equilibria in figure \ref{figCI3}(d).
PDE \eqref{gu} was the starting point of a systematic investigation of connection graphs by Brunovsk\'y and the first author \cite{brfi88,brfi89}.
The notion of connection graphs dates back to Morse theory, essentially, with substantial generalization by Charles Conley \cite{conley}.
See \cite{misch} for an introductory survey. 
For more recent related concepts see \cite{Yorke-b,Yorke-a}.

\begin{figure}[t] 
\begin{center}
\includegraphics[width =  0.95\textwidth]{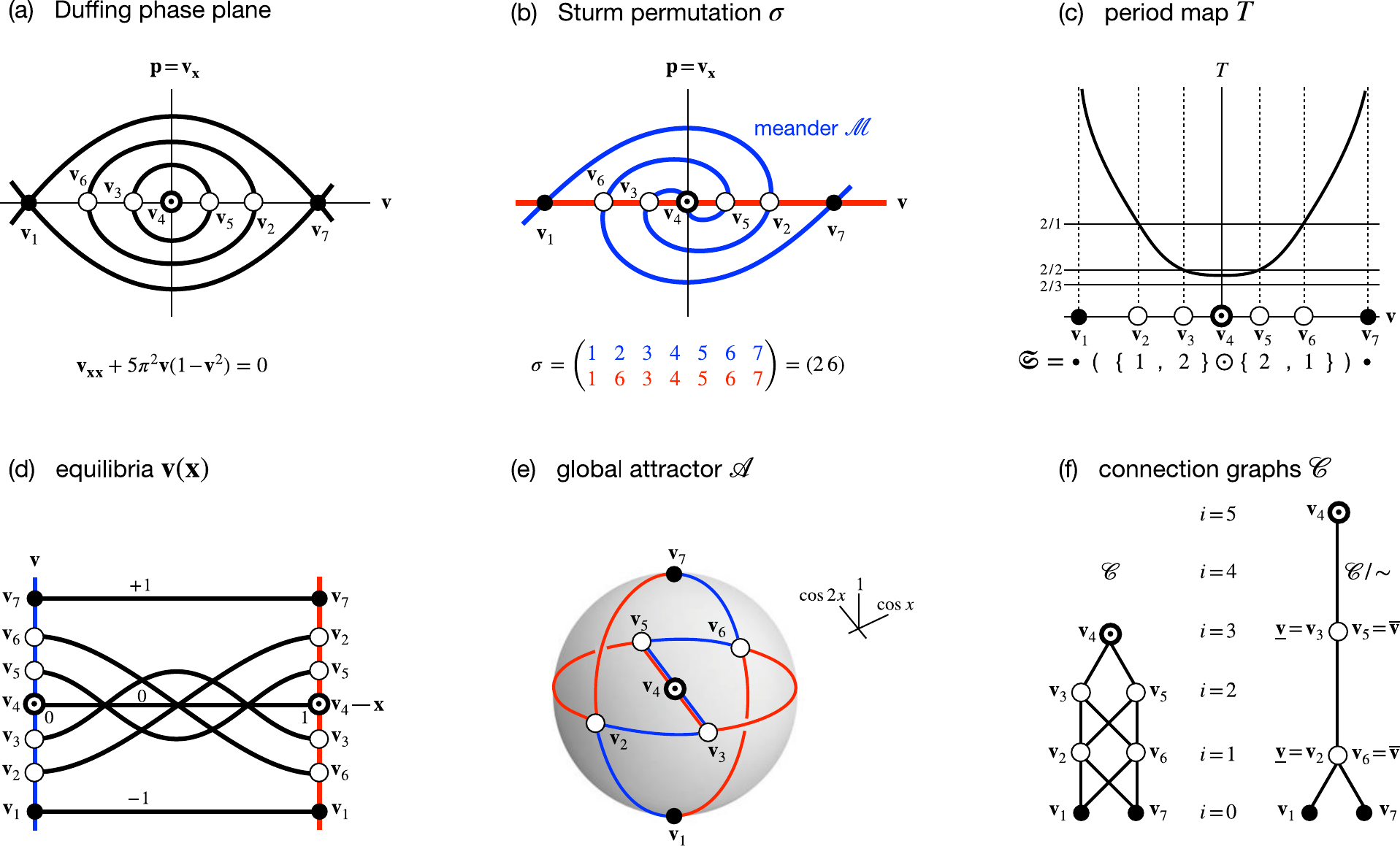}
\end{center}
\caption[The Chafee-Infante attractor $\mathrm{CI}_3$]
{\small\emph{
The three-dimensional Chafee-Infante attractor $\mathrm{CI}_3$ of seven equilibria \cite{chin74, hen81}.
See also case 1.2(i) of table \ref{tablap}.
(a) The ODE phase plane \eqref{gvp} for the Duffing nonlinearity $\mathbf{\mathbf{g}}(\mathbf{u})=\lambda^2 \mathbf{u}(1-\mathbf{u}^2),\ 2\!<\!\lambda/\pi\!<\!3$.
(b) The associated ODE Sturm meander $\cM_\mathbf{g}$ and the Sturm permutation $\sigma_\mathbf{g}=(2\,6)\in S_7$\,.
(c) The time map $T$ and the full lap signature $\fS(T)$.
\hbox{(d) The} spatial equilibrium profiles $\mathbf{v}_j\in\cEgN$. 
Note the ordering by the Sturm permutation $\mathbf{v}_{\sigma_\mathbf{g}(j)}$ at the right boundary $\mathbf{x}=1$.
(e) Sketch of the 3-ball Chafee-Infante attractor $\mathrm{CI}_3=\cAgN$. 
Note the unstable dimensions $i=0$ of the attracting South and North poles $\mathbf{v}_1, \mathbf{v}_7\equiv\pm1$, marked by bullets $\bullet$\,.
The saddles $\underline{\mathbf{v}}=\mathbf{v}_2,\ \overline{\mathbf{v}}=\mathbf{v}_6$ at Morse index $i=1$, and the $i=2$ saddles $\underline{\mathbf{v}}=\mathbf{v}_3,\ \overline{\mathbf{v}}=\mathbf{v}_5$ are indicated by circles $\circ$\,.
The surface 2-sphere is the invariant boundary of the interior, alias of the unstable manifold $W^u(\mathbf{v}_4)$ of the $i=3$ PDE repeller $\mathbf{v}_4$\,.
Explicitly, $\mathbf{v}_4\equiv0=\odot$ is the only ODE center.
(f) The connection graph $\cCgN$ and its reduced variant $\cCgN/\sim$, where pairs $\underline{\mathbf{v}},\overline{\mathbf{v}}$ have been identified.
Vertical ordering is by Morse index $i$.
Edges in the reduced graph $\cCgN/\sim$ indicate that Neumann heteroclinic orbits exist between some equilibrium representatives of its collapsed equilibrium pairs.
}}
\label{figCI3}
\end{figure}

In section \ref{Back} we will recall how \emph{Sturm permutations} $\sigma$, an ingenious device introduced by Fusco and Rocha \cite{furo91}, extended this classification to reaction-advection nonlinearities $\mathbf{g}=\mathbf{g}(\mathbf{x},\mathbf{u},\mathbf{u_x})$.
It simply compares the ordering of the equilibria $\mathbf{v}\!\in\!\cE$ at the left and right Neumann boundaries $\mathbf{x}\!=\!0$ and $\mathbf{x}\!=\!1$, as follows.
 We first label equilibria $\cE=\{\mathbf{v}_1,\ldots,\mathbf{v}_N\}$ by their ordering at $\mathbf{x}=0$:
\begin{equation} \label{sigma0}
\mathbf{v}_1 < \mathbf{v}_2 < \dots < \mathbf{v}_N\,,\qquad \textrm{at } \mathbf{x}=0\,.
\end{equation}
The ordering of the same equilibria at the other Neumann boundary $\mathbf{x}\!=\!1$ will be different, in general.
This defines a \emph{Sturm permutation} $\sigma\!=\!\sigma_\mathbf{g}$ such that
\begin{equation} \label{sigma1}
\mathbf{v}_{\sigma(1)} < \mathbf{v}_{\sigma(2)} < \dots < \mathbf{v}_{\sigma(2n+1)}\qquad \textrm{at } \mathbf{x}=1\,.
\end{equation}
See figures \ref{figCI3}(b) and (d).

The central importance of the Sturm permutation $\sigma$ arises, because it determines the global attractor $\cA$, up to $C^0$ orbit equivalence \cite{firo00}, and the connection graph $\cC$ \cite{firo96}.
See figures \ref{figCI3}(e) and (f), left.

Although \emph{Sturm permutations} can be characterized in purely combinatorial terms, which do not involve any ODE or PDE analysis \cite{firo99}, their counting combinatorics remains elusive.
The Sturm permutations form a subclass of \emph{meander permutations}, which were introduced in \cite{arn89}.
See figure \ref{figCI3}(b) for a first example, and \cite{kar17} for a beautiful introduction, with many illustrations.
Even partial countings of meander permutations, in their own right, involve very advanced combinatorial challenges; see \cite{zograf,lazv04} and the survey \cite{zografICM}.

In our present paper, we return to the even more restrictive class of dissipative nondegenerate reaction-diffusion equations \eqref{gu}.
Already the original analysis \cite{brfi88,brfi89} had been based on equilibria $\mathbf{v}(\mathbf{x})$ and the period maps of their associated ODEs; see section \ref{Sign}.
Equilibria solve the second order pendulum equation
\begin{equation}
\label{gv}
\mathbf{v_{xx}} + \mathbf{g}=0
\end{equation}
with integrable Hamiltonian $\mathbf{H}=\tfrac{1}{2}\mathbf{v}_\mathbf{x}^2+\mathbf{\mathbf{G}}(\mathbf{v})$ and potential energy $\mathbf{G}'(\mathbf{v})=\mathbf{g}(\mathbf{v})$.
Therefore, we call the subclass of Sturm permutations $\sigma$ arising from PDE \eqref{gu}, with a Hamiltonian $\mathbf{H}$, \emph{Hamiltonian Sturm permutations}.

In theorem \ref{thmHamrev} below, we recall a purely combinatorial characterization of Hamiltonian Sturm permutations, as integrable Sturm involutions.
The counting combinatorics of Sturm permutations and meanders remains elusive, in general. 
In the present paper, however, 
\begin{center}
\textbf{it is our main goal to provide detailed results on the counting combinatorics of the restricted class of Hamiltonian Sturm permutations.}
\end{center}
See the theorems at the end of this introduction, for our \emph{main results}.

To prepare, we have to look at the peculiarities of Neumann PDE equilibria in the Hamiltonian ODE setting \eqref{gvp} with nonlinearities $\mathbf{g}\!=\!\mathbf{g}(\mathbf{u})$, in some more detail.
Trivial, spatially homogeneous PDE equilibria are the zeros $\mathbf{e}$ of $\mathbf{g}$, i.e. $\mathbf{v}(\mathbf{x})\equiv\mathbf{e}$.
By hyperbolic nondegeneracy, all zeros are simple: $\mathbf{g}'(\mathbf{e})\neq0$.
We call $\mathbf{e}$ a \emph{center} if $\mathbf{g}'(\mathbf{e})>0$.
In the pendulum $(\mathbf{v},\mathbf{p})$-plane
\begin{equation}
\label{gvp}
\begin{aligned}
    \mathbf{v}_\mathbf{x}&=\mathbf{p} \,,  \\
    \mathbf{p}_\mathbf{x}&=-\mathbf{g}  \,,
\end{aligned}
\end{equation}
centers $(\mathbf{v},\mathbf{p})=(\mathbf{e},0)$ are surrounded by nonstationary periodic ODE orbits of constant energy $\mathbf{H}$.
On the horizontal $\mathbf{v}$-axis, centers alternate with \emph{ODE-saddles} $(\mathbf{e},0)$ where $\mathbf{g}'(\mathbf{e})<0$.
Let $r$ count the number of centers.
By dissipativity \eqref{disign}, we obtain an odd number
\begin{equation}
\label{r}
2r+1
\end{equation}
of spatially homogeneous PDE equilibria.
Throughout, we skip the trivial case $r\!=\!0$ of a globally attracting PDE equilibrium $\cA\!=\!\cE\!=\!\{\mathbf{e}\}$ and assume the presence of $r\geq1$ ODE-centers.
For an example with $r\!=\!1$ see the Duffing oscillator of figure \ref{figCI3}(a).

Consider spatially non-homogeneous PDE equilibria $\mathbf{v}(\mathbf{x})$ next.
 By reversibility of the pendulum \eqref{gv}, \eqref{gvp} in $\mathbf{x}$, we may reflect $\mathbf{v}(\mathbf{x})$ through the Neumann boundaries at $\mathbf{x}=0,1$.
 Then $\mathbf{v}(\mathbf{x})$ becomes spatially periodic with period 2.
 The minimal period $T$ is an integer fraction 
\begin{equation}
\label{Tell}
T=2/\ell>0\,,
\end{equation}
for some \emph{lap number} $\ell>0$.
Hyperbolicity of equilibria is equivalent to $2/\ell$ being regular values of the period map $T$.
For odd $\ell$, the values $\mathbf{v}(\mathbf{x})$ at the Neumann boundaries $\mathbf{x}=0,1$ represent both min/max extrema of $\mathbf{v}$.
For even $\ell$, i.e. when 1 itself is a period, the same extremum appears at both boundaries.
In either case, we see how non-homogeneous PDE equilibria $\mathbf{v}(\mathbf{x})$ are grouped in pairs 
\begin{equation}
\label{vpair}
\mathbf{\underline{v}}\,,\ \mathbf{\overline{v}}\,.
\end{equation} 
Here $\mathbf{\underline{v}}$ attains its minimum at the left boundary $\mathbf{x}\!=\!0$, and $\mathbf{\overline{v}}$ its maximum.
The two variants are shifted by their half-period $T/2$ in $\mathbf{x}$.
Such phenomena have been called a \emph{hidden symmetry} by \cite{gomasc84}.
Periodic boundary conditions are more elucidating here; see section \ref{Per}.

\begin{figure}[t] 
\begin{center}
\includegraphics[width = \textwidth]{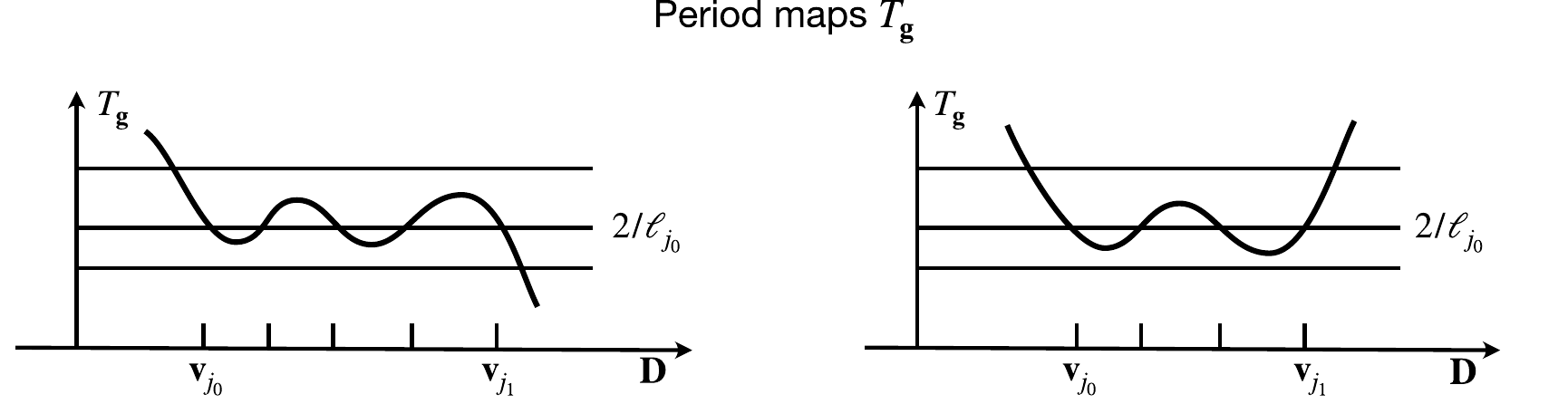}
\end{center}
\caption[Period maps]{\small\emph{ Examples of two period maps $T=T_\mathbf{g}(\mathbf{v}_0)$, partial view.
Horizontal axis: $\mathbf{v}_0\in\mathbf{D}:=\mathbf{C}\cap\{\mathbf{p}=0\}$\,. Vertical axis: $T$.
We assume that all values $T=2/\ell,\ \ell\in\bN$, are regular (horizontal lines),
i.e. the period maps $T$ cross such levels transversely, alternatingly up or down. 
In each connected component of $\mathbf{D}$, and at fixed $\ell$, each crossing indicates boundary values of the same type, say of local minimum type $\mathbf{v}_0=\mathbf{\underline{v}}(0)$.
Note the odd number of consecutive crossings at locally transient mid-level $\ell$ (left), and the even number of crossings at locally extremal $\ell$ (right).}}
\label{figTg}
\end{figure}

In \cite{rofiln} we have classified Hamiltonian Sturm permutations and, thereby, PDE global attractors $\cA$ and their connection graphs $\cC$ of \eqref{gu}, based on a schematic description of the \emph{period maps} 
\begin{equation}
\label{Tmap}
T=T_\mathbf{g}\!: \quad \mathbf{C}\rightarrow (0,\infty)
\end{equation}
associated to the pendulum ODE \eqref{gv}, \eqref{gvp}.
Here the \emph{cyclicity set} $(\mathbf{v},\mathbf{p})\!\in\!\mathbf{C}\!\subset\!\bR^2$ denotes the points on nonstationary periodic orbits of \eqref{gvp}.
See figures \ref{figCI3}(c) and \ref{figTg} for examples.
In section \ref{Sign} we recall the schematic description of period maps $T$ by \emph{full lap signatures} $\fS(T)$, as introduced in \cite{rofiln}.
That description will enable us to address the following counting questions.

Let $r\!\geq\!1$ count the ODE-centers, and $q\!\geq\!0$ the periodic orbits with arbitrary minimal period $T\!=\!2/\ell$, respectively. 
See \eqref{r} and \eqref{Tell}.
Then we have just seen that the total number $|\cE|$ of hyperbolic equilibria $\mathbf{e},\mathbf{v}\!\in\!\cE$ of PDE \eqref{gu} is odd. 
Indeed,
\begin{equation}
\label{n}
|\cE|=N=2n+1\,,\quad\textrm{where}\quad n=r+q\,.
\end{equation}

Let $a_n$ count the global attractors $\cA_\mathbf{g}$ with $2n+1\geq3$ equilibria, as distinguished by their Hamiltonian Sturm permutations $\sigma_\mathbf{g}$ of \eqref{sigma0}, \eqref{sigma1}, alias by their full lap signatures $\mathfrak{S}(T_\mathbf{g})$ defined in section \ref{Sign}.
In section \ref{Asym}, we determine the explicit \emph{generating function} 
\begin{equation}
\label{az}
a(z)=\sum_{n\geq1}\, a_n\,z^n
\end{equation}
of the counts $a_n$\,.
See also corollary \ref{coraz}.
For $n\!\rightarrow\!\infty$, this will provide explicit leading asymptotics of $a_n$\,, in theorem \ref{thmasa}.

Next let $b_{rq}$ denote the refined count of Hamiltonian Sturm permutations $\sigma_\mathbf{g}$\,, alias full lap signatures $\mathfrak{S}(T_\mathbf{g})$. 
Here $r$ fixes the number of ODE-centers, i.e. of the $2r+1$ spatially homogeneous PDE equilibria. The second index, $q$, fixes the number of pairs $\mathbf{\underline{v}}\,,\ \mathbf{\overline{v}}$ of spatially non-homogeneous equilibria.
See \eqref{r}, \eqref{vpair}.
Since $r\!+\!q\!=\!n$, fixing $n$ implies 
\begin{equation}
\label{ab}
a_n\ =\ \sum_{r=1}^n \ b_{r,\,n-r}\ .
\end{equation}
We also define the refined generating function 
\begin{equation}
\label{bxy}
b(x,y)=\sum_{q\geq0}\sum_{r\geq1}\, b_{rq}\,x^ry^q\,.
\end{equation}

\begin{theorem}\label{thmbxy}
The refined generating function $b$ of Hamiltonian Sturm permutations is
\begin{equation}
\label{bb0}
b(x,y)=b_0/(1-b_0)\,.
\end{equation}
Here we have to substitute the explicit generating function
\begin{equation}
\label{b0}
b_0 = \frac{y^2}{1-\sqrt{1-4 y^2}}\cdot
\left(1+\tfrac{2x}{1-2y+\sqrt{1-4y^2}}-
\sqrt{1+4x\cdot
\tfrac{2-y(3-x+2y) -(2-y)\sqrt{1-4 y^2}}{y(1-2y+\sqrt{1-4y^2})^2}
}
\right).
\end{equation}
\end{theorem}

Note analyticity of the rightmost fraction, at $y\!=\!0$.
We prove theorem \ref{thmbxy} in section \ref{Pfb}.

Aiming for the generating function $a(z)$ of the Hamiltonian Sturm counts, we now rewrite the refined generating function $b(x,y)$ in terms of new labeling variables 
\begin{equation}
\label{uz}
\begin{aligned}
    x&\phantom{:}=uz\,;   \\
    y&\phantom{:}=z\,; \\
    f(u,z)&:=b(x,y)=b(uz,z)\,.
\end{aligned}
\end{equation}
This rearranges the counting coefficients of $b$ as
\begin{equation}
\label{fnr}
\begin{aligned}
    f(u,z)&=\sum_{n=1}^\infty \left(\sum_{r=1}^n f_{nr}\,u^r\right)z^n=\sum_{n=1}^\infty f_n(u)z^n\,,\quad\textrm{with}  \\
    f_{nr}&=b_{r,n-r}
\end{aligned}
\end{equation}
Here $f_{nr}$ counts the Hamiltonian Sturm permutations with a total of $2n+1$ PDE equilibria and $r\geq1$ ODE-centers.
Explicitly, this rewrites theorem \ref{thmbxy} as $f(u,z)=f_0/(1-f_0)$ with the abbreviation
\begin{equation}
\label{f0}
f_0= \frac{z^2}{1-\sqrt{1-4 z^2}}\cdot
\left(1+\tfrac{2uz}{1-2z+\sqrt{1-4z^2}}-
\sqrt{1+4u\cdot
\tfrac{2-z(3-uz+2z) -(2-z)\sqrt{1-4 z^2}}{(1-2z+\sqrt{1-4z^2})^2}
}
\right).
\end{equation}

Evaluating $f(u,z)$ in \eqref{uz}, \eqref{fnr} at $u\!=\!1$ immediately yields the generating function $a(z)$ of the attractor counts $a_n$ in \eqref{az}, \eqref{ab} as follows.
\begin{corollary}\label{coraz}
The generating function $a(z)$ of Hamiltonian Sturm permutations satisfies 
\begin{equation}
\label{aexplic}
a(z)=f(1,z)=b(z,z)\,.
\end{equation}
Explicitly, this identifies $a(z)=a_0/(1-a_0)$ with the substituted generating function
\begin{align}
\label{a0}
a_0(z)&:= \frac{z^2}{1-\sqrt{1-4 z^2}}\cdot
\left(1+\frac{2z-\sqrt{2\,\mathfrak{p}_*(z)}}{1-2z+\sqrt{1-4z^2}}\right)\,,\quad\textrm{and}
\\[2mm]
\label{fpz}
\mathfrak{p}_*(z)&:=5-3 \sqrt{1 - 4 z^2} -8z-2z^2\,.
\end{align}
\end{corollary}

The corollary suggests to normalize the counts $f_{nr}$ and consider the polynomials
\begin{equation}
\label{pnu}
\begin{aligned}
    p_n(u)&:= \sum_{r=1}^n p_{nr}\,u^r\,,\quad\textrm{with} \\
    p_{nr}&:=f_{nr}/a_n =b_{r,n-r}/a_n\,. 
\end{aligned}
\end{equation}
as generating functions of random variables $X_n$ with probabilities $\bP(X_n\!=\!r):=p_{nr}$\,.
As before, $f_{nr}\geq0$ count Hamiltonian Sturm permutations $\sigma$ with $r$ centers and $q\!=\!n\!-\!r$ pairs $\underline{\mathbf{v}},\ \overline{\mathbf{v}}$ of non-homogeneous equilibria.
The $p_{nr}$ normalize $f_{nr}$ by the sum $a_n=\sum_{r}f_{nr}$ to qualify as probabilities.
The probabilities $p_{nr}$ indicate the relative frequencies of full lap signatures $\cS$ with $r$ ODE-centers among those with $|\cE|\!=\!N\!=\!2n\!+\!1$ PDE equilibria.
We define means and variances, as usual:
\begin{equation}
\label{mu}
\begin{aligned}
    \boldsymbol{\mu}_n&:=\sum_{r=1}^n\, r\,p_{nr}\quad\textrm{and}   \\
    \boldsymbol{\sigma}_n^2&:= \sum_{r=1}^n\, (r-\boldsymbol{\mu}_n)^2\,\,p_{nr}\,,
\end{aligned}
\end{equation}
By $a_{0n},\  b_{0,rn},\  f_{0,nr},\ p_{0,nr},\  X_{0n},\  \boldsymbol{\mu}_{0n},\  \boldsymbol{\sigma}_{0n}$ we denote the corresponding quantities, based on the explicit generating functions $b_0(x,y),\  f_0(u,z),\ a_0(z)$ from \eqref{b0}, \eqref{f0}, \eqref{a0}, instead.
In section \ref{Pfb} it will turn out how all these quantities refer to the counting of irreducible Hamiltonian Sturm permutations, alias irreducible full lap signatures, alias signature trees.

Under appropriate normalization by means and variances, we obtain the following Gaussian limit theorem.

\begin{theorem}\label{thmGauss}
For $n\!\rightarrow\!\infty$, the normalized random variables 
\begin{align}
\label{Y0n}
Y_{0n}&:=(X_{0n}-\boldsymbol{\mu}_{0n})/\boldsymbol{\sigma}_{0n}\quad\textrm{and}\\
\label{Yn}
Y_n&:=(X_n-\boldsymbol{\mu}_n)/\boldsymbol{\sigma}_n
\end{align}
converge to the standard Gaussian.
Convergence of the distribution functions $\bP(Y_{0n}\leq y)$ and $\bP(Y_n\leq y)$ is uniform in $y$, of order $\cO(1/\sqrt{n})$.
\end{theorem}

Theorem \ref{thmGauss} is proved in section \ref{Gauss}.
In particular, this implies an asymptotically Gaussian shape of the normalized relative frequencies
\begin{equation}
\label{b/a}
p_{nr}:=b_{rq}\,/\,a_n\,,\qquad \textrm{for } n=r+q\rightarrow\infty\,.
\end{equation}
The following asymptotics for means $\boldsymbol{\mu}_n$ and variances $\boldsymbol{\sigma}_n^2$ are also proved in section \ref{Gauss}.
The results are based on detailed asymptotics of $a_n$\,, developed in section \ref{Asym}.
Both theorems crucially rely on the analytic combinatorics of \cite{flajolet}.

\begin{theorem}\label{thmbasy}
For means and variances in the limit  $n\rightarrow\infty$, we obtain the asymptotics
\begin{align}
\label{muas}
    \boldsymbol{\mu}_n&=\mathfrak{m}_1\,n+\mathfrak{m}_0+\cO_{\mu}(\tfrac{1}{n})\,;  \\
\label{varas}
    \boldsymbol{\sigma}_n^2&=\,\mathfrak{v}_1\,n\,+\,\mathfrak{v}_0+\cO_{\sigma^2}(\tfrac{1}{n})\,.
\end{align}
Numerical values for the universal coefficients are
\begin{align}
\label{m1num}
\mathfrak{m}_1 &= 0.367\,986\,760\,847\,991 \ldots\,;\\ 
\label{m0num}
\mathfrak{m}_0 &= 0.566\,906\,944\,822\,056 \ldots\,;   \\
\label{v1num}
\mathfrak{v}_1\, &= 0.186\,798\,320\,509\,941 \ldots\,;    \\ 
\label{v0num}
\mathfrak{v}_0\, &= 0.282\,135\,612\,383\,251\ldots\,.   
\end{align}
For $\boldsymbol{\mu}_{0n}$ and $\boldsymbol{\sigma}_{0n}$ we obtain analogous asymptotics.
Their leading terms $\mathfrak{m}_{0,1}=\mathfrak{m}_1$ and $\mathfrak{v}_{0,1}=\mathfrak{v}_1$ coincide with \eqref{m1num} and \eqref{v1num}.
The constant terms are modified to become
\begin{align}
\label{m00num}
\mathfrak{m}_{0,0} &= -0.264\,839\,040\,868\,747\ldots\,;   \\
\label{v00num}
\mathfrak{v}_{0,0}\, &= -0.039\,686\,978\,502\,604\ldots\,.   
\end{align}
\end{theorem}

See section \ref{Num} for some numerical illustrations of the asymptotics of the Hamiltonian Sturm counts $a_n$ from theorem \ref{thmasa}, and of the Gaussian asymptotics of the relative frequencies $p_{nr}$ from theorems \ref{thmGauss}, \ref{thmbasy}.

As we have announced before, the following section \ref{Back} comments on the PDE background of our counting problem.
Section \ref{Sign} then presents the classification of Hamiltonian Sturm permutations $\sigma$ by full lap signatures; see definition \ref{defSign}.
Borrowing from \cite{fiedlertatra,firo96,firowo12b,rofiln}, we discuss the 21 cases $n=1,2,3,4$ in section \ref{Ex}; see table \ref{tablap}.

In section \ref{Per}, we conclude with the case of periodic boundary conditions $\mathbf{x}\!\in\!\bS^1\!=\!\bR/2\bZ$.
Freezing rotating waves, we proceed analogously to the Neumann case, except that ODE periodic orbits give rise to $q$ frozen PDE waves, instead of $q$ spatially non-homogeneous Neumann pairs \eqref{vpair} of equilibria.
The combinatorial analysis of the spatially periodic case extends to more general spatially homogeneous nonlinearities $\mathbf{g}=\mathbf{g}(\mathbf{u},\mathbf{u_x})$.
Under Neumann boundary conditions, we additionally require spatial reversibility, i.e.\ $\mathbf{g}(\mathbf{u},-\mathbf{p})=\mathbf{g}(\mathbf{u},\mathbf{p})$ for all $\mathbf{u},\mathbf{p}$, in order to ensure Hamiltonian Sturm permutations $\sigma_\mathbf{g}$\,.

\section{Background: Sturm permutations and global attractors}\label{Back}

To survey the PDE background of our counting problem, we slightly generalize the reaction-diffusion PDE \eqref{gu} to
\begin{equation}
\label{gxuux}
\mathbf{u_t} = \mathbf{u_{xx}} + \mathbf{g}(\mathbf{x},\mathbf{u},\mathbf{u_x})\,,
\end{equation}
which admits  explicit spatial dependence $\mathbf{x}$ and nonlinear advection $\mathbf{u_x}$\,.
For now, we stick with Neumann boundary condition \eqref{Nbc}.
We assume the nonlinearity $\mathbf{g}{:}\ [0,1]\times\bR^2\rightarrow \bR$ to be $C^2$-smooth, for simplicity. 
Then the initial value problem \eqref{gxuux} generates a local semiflow of solutions, e.g. on suitable fractional Sobolev spaces 
$\mathbf{u}(\mathbf{t},\cdot)\!\in\! \mathbf{X}\!\subset\! H^s\!\hookrightarrow\! C^1,\ 3/2\!<\!s\!<\!2$, with Neumann boundary.
See for example \cite{hen81,paz83}, and further comments in section \ref{Back}.
In section \ref{Per} we will also discuss the analogous setting for periodic boundary conditions \eqref{pbc}.

\emph{Dissipativity} of the nonlinearity $\mathbf{g}$ still requires any solution to enter and remain in a fixed large ball in $\mathbf{X}$, for all sufficiently large times.
Sufficient conditions for dissipativity, which we impose throughout, are 
\begin{equation}
\label{diss}
\mathbf{g}(\mathbf{x},\mathbf{u},0)\cdot\mathbf{u}<0\quad \textrm{and} \quad |\mathbf{g}(\mathbf{x},\mathbf{u},\mathbf{p})|\leq C(|\mathbf{u}|)(1+\mathbf{p}^2)\,,
\end{equation}
for all large $|\mathbf{u}|$ and some continuous function $C(|\mathbf{u}|)$; see \cite{mana97}.
Compare also \eqref{disign}.

Under dissipativity, our previous general remarks on compact global attractors $\cA$ apply.
Surprisingly, even the gradient-like structure persists for suitable Lyapunov functionals $V$.
This was originally observed by Zelenyak \cite{zele}; see also \cite{mat88,firaro,lafi,labe} for further generalizations.

As before, we assume spectral nondegeneracy of $\mathbf{g}(\mathbf{x},\mathbf{u},\mathbf{u_x})$, i.e. we assume all equilibria $\mathbf{v}$ to be \emph{hyperbolic}.
The global attractor $\cA$ still consists of an odd number of $N\!=\!2n+1$ equilibria $\mathbf{v}\!\in\!\cE$, and of heteroclinic orbits $\mathbf{u}(\mathbf{t},\cdot)\!\in\!\cH$ as in \eqref{het}, between them.
Again, this associates a connection graph $\cC$ to PDE \eqref{gxuux}.

Since we consider dissipative nonlinearities $\mathbf{g}$ with hyperbolic equilibria and various dependencies on $\mathbf{x},\mathbf{u},\mathbf{u_x}$\,, we introduce the following notation for such classes of $\mathbf{g}$. 
Superscripts $\cN$ refer to Neumann  boundary conditions on $\mathbf{x}\!\in\![0,1]$.
We defer the quite analogous case $\cP$ of periodic boundary conditions \eqref{pbc} to our concluding section \ref{Per}.
Section \ref{Intro}, theorem \ref{thmbxy}, corollary \ref{coraz}, and theorem \ref{thmGauss} were based on the \emph{Hamiltonian class} $\mathbf{g}\!=\!\mathbf{g}(\mathbf{u})$, which we denote as
\begin{equation}
\label{SHN}
\mathbf{g}\in\SHN\,.
\end{equation}
The present section has been set up to address
\begin{equation}
\label{SxN}
\mathbf{g}\in\SxN\,.
\end{equation}
In later sections we will also classes of nonlinearities like
\begin{align}
\label{SNchain}
   \mathbf{g}\in \SH\subset\SO&\subset\SSON\subset\SxN\,,\quad\textrm{and}  \\
\label{SPchain}
   \mathbf{g}\in \SH\subset\SO&\subset\SSOP\subset\SxP\,.
\end{align}
The notation $\mathbf{u}_\mathbf{x}^2$ indicates that $\mathbf{g}(\mathbf{u},\mathbf{p})$ is assumed to be even in $\mathbf{p}$.
The two leftmost classes of nonlinearities, it turns out, do not depend on the boundary conditions.
We have therefore omitted the superscripts $\cN,\cP$ there.

To emphasize dependence on nonlinearities and boundary conditions for equilibria, heteroclinic orbits, global attractors, and connection graphs, we will also use self-explanatory notation like
\begin{align}
\label{EgN}
    &\cEgN\,,\ \cHgN\,,\ \cAgN\,,\ \cCgN\,,\ \quad\textrm{and}  \\
\label{EgP}
    &\cEgP\,,\ \cHgP\,,\ \cAgP\,,\ \cCgP\,.
\end{align}

We now address global attractors $\cA\!=\!\cAgN$ for PDE \eqref{gxuux}, i.e. for $\mathbf{g}\!\in\!\SxN$, more specifically.
It was a main result of \cite{firo96,firo00}, how the PDE connection graph $\cC\!=\!\cCgN$ is determined, explicitly and constructively, by the Sturm permutation $\sigma$ of the boundary ordering \eqref{sigma0}, \eqref{sigma1} of the $N\!=\!2n+1$ equilibria $\mathbf{v}_1,\ldots,\mathbf{v}_{2n+1}\in\cE$.
This includes non-Hamiltonian nonlinearities $\mathbf{g}$.
Still, the (not necessarily Hamiltonian) Sturm permutations $\sigma$ only require input from the equilibrium ODE \eqref{gv}.

A special case arises, when we just concatenate two Sturm permutations $\sigma_1$ and $\sigma_2$ at a shared index $N_*$\,.
In other words, we require the restrictions $\sigma_1$ and $\sigma_2$ of the joint Sturm permutation $\sigma$ to $\{1,\ldots,N_*\}$ and to $\{N_*,\ldots,N\}$ with $N:=2n+1$, respectively, to be Sturm.
In particular, $N_*=2n_*+1\neq0,N$ needs to be odd.
This amounts to stacking the associated Sturm global attractors $\cA_1$ and $\cA_2$ on top of each other, glued at the shared, and separating, equilibrium 
\begin{equation}
\label{stack}
\cA_1\leq\mathbf{v}_{N_*}\leq\cA_2\,.
\end{equation}
Here each attractor, separately,  would be depicted in the style of figure \ref{figCI3}.
Several examples of such ``gluing'' were described in \cite{fiedlertatra}; see also ``stacking'' in \cite{firo09} and in section 2.3 of \cite{firo10}.
If such stacking does not occur, we call $\sigma$ an  \emph{irreducible Sturm permutation}.

The perhaps most powerful PDE feature of \eqref{gxuux} is the dropping behavior of the \emph{zero number} 
$\mathbf{z}:\ \mathbf{X}\!\setminus\!\{0\}\!\rightarrow\!\{0,1,2,\ldots\}$\,.
Specifically, $\mathbf{z}(\mathbf{w})$ counts the number of strict sign changes of $\mathbf{x}\!\mapsto\!\mathbf{w}(\mathbf{x})\!\in\!\mathbf{X}\!\setminus\!\{0\}\!\subset\!C^1$.
Let $\mathbf{w}(\mathbf{t},\mathbf{x}):=\mathbf{u}_2(\mathbf{t}\mathbf{,}x)\!-\!\mathbf{u}_1(\mathbf{t},\mathbf{x})$ denote the difference of any two nonidentical solutions $\mathbf{u}_1,\,\mathbf{u}_2$ of \eqref{gxuux}. 
Then 
\begin{equation} 
\label{zdrop}
\mathbf{t} \mapsto \mathbf{z}(\mathbf{w}(\mathbf{t},\cdot)) \searrow
\end{equation} 
is nonincreasing with $\mathbf{t}$.
More precisely, $\mathbf{z}$ becomes finite, as soon as $\mathbf{t}>0$, and drops strictly at any $\mathbf{t}_0>0$ for which $\mathbf{x}\mapsto \mathbf{w}(\mathbf{t}_0,\mathbf{x})$ possesses a multiple zero $\mathbf{w}\!=\!\mathbf{w_x}\!=\!0$ at some $\mathbf{x}\!=\!\mathbf{x}_0\!\in\!\ [0,1]$.
See \cite{ang88, mat82}, with the original idea going back to Charles-François Sturm \cite{sturm} in 1836.
This is why we attach the name of \emph{Sturm} to the PDEs \eqref{gu} or \eqref{gxuux}, their boundary permutations $\sigma$, global attractors $\cA$, and connection graphs $\cC$.

Most importantly, zero number dropping \eqref{zdrop} implies \emph{automatic transversality} of stable and unstable manifolds
\begin{equation}
\label{transv}
\mathbf{W^u}(\mathbf{v}_-)\transv \mathbf{W^s}(\mathbf{v}_+)\,,
\end{equation}
for any pair of hyperbolic equilibria $\mathbf{v}_\pm\in\cE$.
See \cite{hen85, czro08, firowo04, jora10a, roc25}.

As a first consequence of transversality, the heteroclinicity relation $\mathbf{v}_-\leadsto \mathbf{v}_+\,$ is transitive.
This follows from the celebrated $\lambda$-Lemma \cite{PalisdeMelo}.
In the transitive connection graph $\cCgN$, we may therefore omit heteroclinic edges which arise from concatenations of other heteroclinic edges.
Then any directed path from $\mathbf{v}_1$ to $\mathbf{v}_2$ also implies a direct heteroclinic orbit $\mathbf{v}_1\leadsto \mathbf{v}_2$.
Given $\mathbf{v}_1\leadsto \mathbf{v}_2$\,, conversely, the cascading principle first described in \cite{brfi89} asserts an interpolating sequence of heteroclinic orbits between equilibria of adjacent Morse indices $i$, from $\mathbf{v}_1$ to $\mathbf{v}_2$\,. 
Therefore, all connection graphs $\cCgN$ are graded by the Morse indices $i$, and all heteroclinic edges run between adjacent Morse levels, downward.

As a second consequence of transversality \eqref{transv}, hyperbolicity of equilibria, alone, already identifies the dynamics on any Sturm global attractor $\cA$ as a Morse-Smale system.
In particular, this implies local $C^0$  structural stability.
See \cite{Palis, PalisSmale, PalisdeMelo,hamaol02} for a general background on transversality and Morse-Smale systems, including infinite-dimensional settings.
Perhaps even more surprisingly, the decomposition 
\begin{equation}
\label{AWu}
\cA=\bigcup_{\mathbf{v}\in\cE}\ \mathbf{W^u}(\mathbf{v}))
\end{equation}
into unstable manifolds identifies the global attractor $\cA$ as a signed regular cell complex \cite{firo22,rofi24}.

Under Neumann boundary conditions, the mere ODE Sturm permutation $\sigma\!=\!\sigma_\mathbf{g}$ of any nonlinearity $\mathbf{g}\!\in\!\SxN$ in fact determines the associated PDE global attractor $\cA\!=\!\cA_\mathbf{g}^\cN$, up to $C^0$ orbit equivalence. 
In particular, this also determines the  associated connection graph $\cC\!=\!\cC_\mathbf{g}^\cN$, up to orientation preserving graph isomorphism.
The main reason is that $\sigma_\mathbf{g}$ determines all PDE Morse indices $i_j\!=\!i(\mathbf{v_j})$ and all zero numbers $z_{jk}\!=\!z(\mathbf{v}_j-\mathbf{v}_k)$, by explicit formulas.
See \cite{firo96, firo00} for details and figure \ref{figCI3} for illustration.

The purely combinatorial characterization of all $\sigma_\mathbf{g}$ identifies the Sturm permutations to be dissipative Morse meanders \cite{firo99}.
Here \emph{dissipative}, first, means that the permutation $\sigma_\mathbf{g}\!\in\! S_{2n+1}$ fixes the lowest entry $1$ and the highest entry $2n+1$.
Second, the \emph{Morse} property requires $i_j\geq0$, for all formally associated ``Morse'' numbers $i_j$\,.
\emph{Meanders} $\cM$ are $C^1$ planar Jordan curves which cross the horizontal axis transversely at a finite number of points, emanating from SouthWest and leaving towards NorthEast, eventually.
The \emph{meander} property, third, then requires $\sigma_\mathbf{g}$ to originate from the different enumerations of the crossing points, along the meander and along the horizontal axis.

Necessity of these three conditions is obvious, by dissipativity of $\mathbf{g}$ and because the formal ``Morse'' numbers $i_j$ become true Morse indices of equilibria.
The meander property follows, if we define $\cM\!=\!\cM_\mathbf{g}$ to be the image of the horizontal $\mathbf{v}$-axis at ``time'' $\mathbf{x}\!=\!1$ under the ODE-flow \eqref{gv}.
See figures \ref{figCI3}(a) and (b).

Comparing the counts of ODE Sturm permutations $\sigma_\mathbf{g}$ and of PDE global attractors $\cAgN$ or their connection graphs $\cCgN$, however, encounters subtle difficulties.
Although we have a map $\sigma_\mathbf{g}\mapsto\cAgN$, that map may not be injective.
For example, let 
\begin{equation}
\label{-u}
\mathbf{f}(\mathbf{x},\mathbf{u},\mathbf{p}):=-\mathbf{g}(\mathbf{x},-\mathbf{u},-\mathbf{p})\,.
\end{equation}
Then $\mathbf{u}(\mathbf{t},\mathbf{x})$ solves PDE \eqref{gxuux} if, and only if, $-\mathbf{u}(\mathbf{t},\mathbf{x})$ satisfies the same PDE with nonlinearity $\mathbf{f}$ instead of $\mathbf{g}$.
In particular, the global attractors 
\begin{equation}
\label{-A}
\mathcal{A}_\mathbf{f}^\mathcal{N}\ =\ -\cAgN\ \cong\ \cAgN
\end{equation}
are $C^0$ orbit equivalent, and even linearly flow equivalent.
On the other hand, consider the involutive permutation $\kappa(j):=2n+2-j$ in $S_{2n+1}$\,.
By definition \eqref{sigma0}, \eqref{sigma1}, the Sturm permutations $\sigma_\mathbf{f}$ and $\sigma_\mathbf{g}$ are conjugate under the involution $\kappa$\,:
\begin{equation}
\label{sigmakappa}
\sigma_\mathbf{f}=\kappa\,\sigma_\mathbf{g}\,\kappa\,,
\end{equation}
frequently with $\sigma_\mathbf{f}\neq\sigma_\mathbf{g}$\,.
The Duffing example of figure \ref{figCI3}, in contrast, satisfies $\mathbf{f}\!=\!\mathbf{g}$ and therefore $\sigma_\mathbf{f}\!=\!\sigma_\mathbf{g}$\,.

Passing to $\mathbf{u}(\mathbf{t},1-\mathbf{x})$, we encounter a similar example under reversal of $\mathbf{x}$:
\begin{align}
\label{urev}
\mathbf{f}(\mathbf{x},\mathbf{u},\mathbf{p})&:=\mathbf{g}(1-\mathbf{x},\mathbf{u},-\mathbf{p})\,;\\
\label{Arev}
\mathcal{A}_\mathbf{f}^\mathcal{N}&\phantom{:}\cong\cAgN\,;\\
\label{sigmarev}
\sigma_\mathbf{f}&\phantom{:}=\sigma_\mathbf{g}^{-1}\,.
\end{align}
Although $\sigma_\mathbf{f}\neq\sigma_\mathbf{g}$, in general, we note that the reversible Sturm class $\mathbf{g}\!\in\!\SON\!\supset\!\SHN$ satisfies $\mathbf{f}\!=\!\mathbf{g}$ and therefore $\sigma_\mathbf{g}\!=\!\sigma_\mathbf{f}\!=\!\sigma_\mathbf{g}^{-1}$ is involutive.

Of course we may also achieve the inverse conjugate $\kappa\,\sigma_\mathbf{g}^{-1}\,\kappa$.
In (6.7)--(6.9) of \cite{firo96}, we have called these phenomena \emph{trivial equivalences} of global attractors (and of connection graphs).
Less trivial subtleties of non-injectivity might involve nonlinear $C^0$ orbit equivalences.  
On the positive side, however, we have introduced the signed Thom-Smale complex of Sturm global attractors, in \cite{firo22}.
This is a finer classification than $C^0$ orbit equivalence.
In particular, we have concluded in the last paragraph of \cite{firo22}, how these Sturmian signed regular cell complexes are in 1-1 correspondence with the Sturm permutations $\sigma$, and are therefore counted correctly by our present results.

To avoid such subtleties, in the present paper, we just aim for Sturm permutations.
For general nonlinearities $\mathbf{g}\!\in\!\SxN$ with a prescribed number $N\!=\!2n+1$ of equilibria, however, even just counting Sturm permutations $\sigma\!=\!\sigma_\mathbf{g}$ remains an elusive task.
For an enumerative table up to $n\!=\!20$ see \cite{wol17}. 
We therefore return to the more restrictive class of reversible Sturm nonlinearities $\mathbf{g}\!\in\!\SON\!\supset\!\SHN$ and their Hamiltonian Sturm permutations $\sigma_{\mathbf{g}}$\,, next.

\section{Period maps and full lap signatures}\label{Sign}

In this section we study spatially reversible PDEs
\begin{equation}
\label{guux}
\mathbf{u_t} = \mathbf{u_{xx}} + \mathbf{g}(\mathbf{u},\mathbf{u_x})
\end{equation}
with Neumann boundary. 
Spatial reversibility requires $\mathbf{g}$ to be even in the gradient $\mathbf{p}$:
\begin{equation}
\label{grev}
\mathbf{g}(\mathbf{u},\mathbf{p})=\mathbf{g}(\mathbf{u},-\mathbf{p})\,;\\
\end{equation}
compare \eqref{urev}.
In shorthand \eqref{SNchain}, we generalize section \ref{Intro} and study $\mathbf{g}\!\in\!\SON\!\supset\!\SHN$.
From \cite{rofiln}, we will recall the properties of period maps $T\!=\!T_\mathbf{g}$\,, of their associated full lap signatures $\fS\!=\!\fS_\mathbf{g}$\,, and of Hamiltonian Sturm permutations $\sigma_\mathbf{g}$\,.

The concept of period maps $T\!=\!T_\mathbf{g}$ carries over from the Hamiltonian case $\mathbf{g}\!\in\!\SHN$ to the reversible case $\mathbf{g}\!\in\!\SON$, verbatim, except for the lack of an explicit Hamilton function $\mathbf{H}$.
See \eqref{Tell}--\eqref{Tmap}, as well as figures \ref{figTg} and \ref{figCI3}(c).

Following section 4.3 of \cite{rofiln}, we introduce the notion of a full lap signature $\fS$, independently of period maps.
\begin{definition}\label{defSign}
Let $\fS$ be a finite sequence of symbols $\bullet,\odot,(,)$, and a collection of pairs $\underline{S}_k,\,\overline{S}_k$ of finite ordered \emph{lap lists} $S=\{\ell_1,\ldots,\ell_{s_k}\}$ of positive integers $\ell_j\!\in\!\bN$.
We call $\fS$ a \emph{full lap signature}, if the following nine properties (i)--(ix) all hold.
We first collect the five properties of symbols $\bullet,\odot,(,)$.
\begin{enumerate}[(i)]
	\item The sequence starts and ends with the symbol $\bullet$\,.
	\item The parenthesis symbols \emph{``(''} and \emph{``)''}, by themselves, form a regular parenthesis structure of balanced open-close pairs.
 		Redundant parentheses $\ldots ((\ldots ))\ldots$ are forbidden.
 		We call innermost pairs ``$()$'' \emph{central} and noncentral pairs \emph{annular}.
	\item Adjacent parentheses \emph{``)(''}  take the complete form $\ldots)\,\bullet\,(\ldots$ in $\fS$.
	\item Each central pair \emph{``()''} takes the complete form $\ldots(\underline{S}_k \odot\overline{S}_k)\ldots$ in $\fS$, with a pair of possibly empty lap lists $\underline{S}_k\,,\,\overline{S}_k$\,.
	\item Each annular pair of lap lists takes the complete form $\ldots(\underline{S}_k (\ldots)\,\bullet\ldots\bullet\,(\ldots)\overline{S}_k)\ldots$ in $\fS$, with
		non\-empty lists $\underline{S}_k\,,\,\overline{S}_k$\,.
\end{enumerate}
Any pair of nonempty lap lists $S=\underline{S}_k=\{\ell_1,\ldots,\ell_{s_k}\}$ and $\overline{S}_k$ satisfies all of the following remaining four conditions.
\begin{enumerate}[(vi)]
	\item Each lap list $\overline{S}_k$ is the paired lap list $\underline{S}_k$ in reverse.
	\item \emph{Boundary conditions:}
		\begin{equation} \label{laplistbc}
		\ell_1=1\quad \textrm{and, in case of annular pairs: } \quad  s_k \textrm{ is even, and } \ell_{s_k}=1\,.
		\end{equation}
	\item \emph{Neighbor jump condition:}
		\begin{equation} \label{laplist1}
		|\ell_{j+1}-\ell_j| \le 1 \,,\quad \textrm{for}\quad j=1,\dots,s_k-1 \,. \end{equation}
	\item[(ix)] \emph{Alternate jump condition:} \\ 
		If $\ell_{j_0-1} \ne \ell_{j_0} = \dots = \ell_{j_1} \ne \ell_{j_1+1}$ for $1 \le j_0 \le j_1<s_k$, then
		\begin{equation} \label{laplistalt}
		(\ell_{j_0}-\ell_{j_0-1})(\ell_{j_1+1}-\ell_{j_1})=(-1)^{j_1-j_0} \,.
		\end{equation}
		Here we define $\ell_0:=0$ in case $j_0\!=\!1$. 
		For annular pairs with $j_1\!=\!s_k$ we also require $j_1-j_0$ to be even.
\end{enumerate} 
\end{definition}

\begin{figure}[t] 
\begin{center}
\includegraphics[width = 0.9\textwidth]{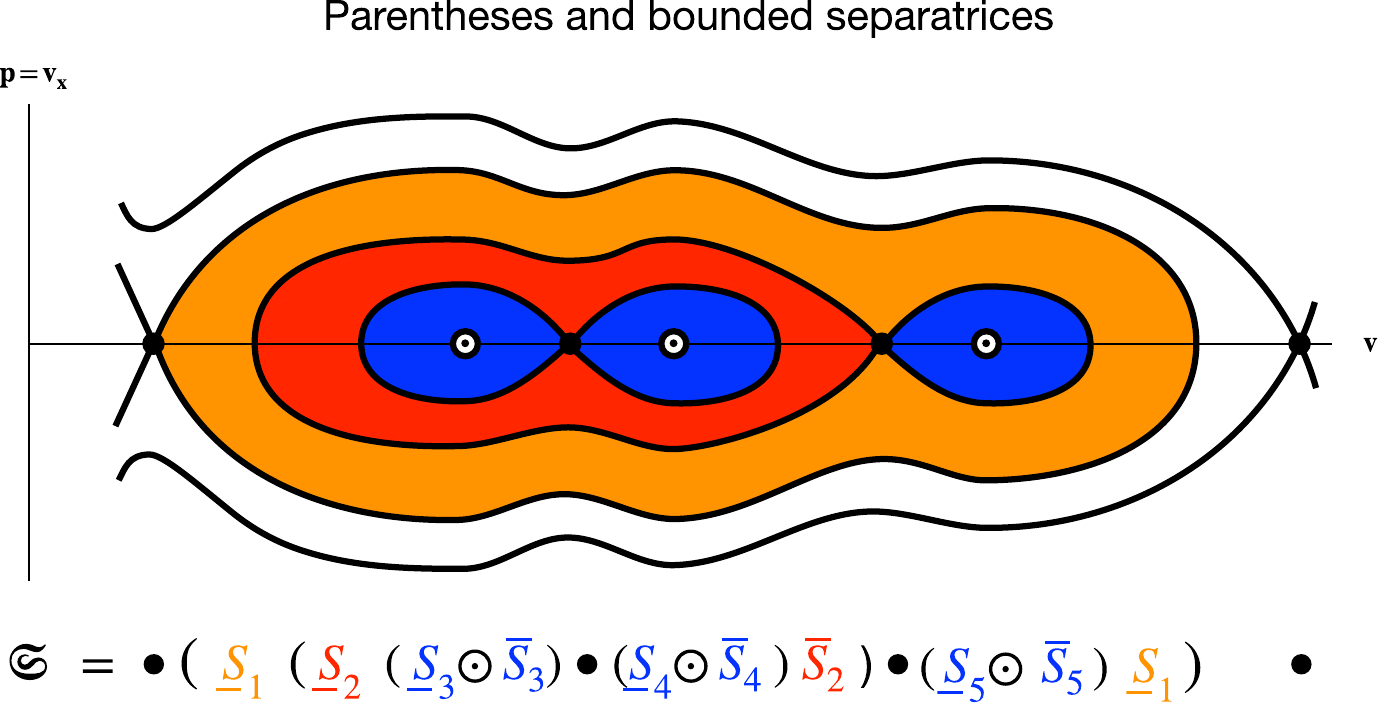}
\end{center}
\caption[Parentheses and bounded separatrices]{\small\emph{Example of the saddle separatrices in the phase plane \eqref{gvp} for a reversible nonlinearity $\mathbf{g}\in\SON$.
Separatrices emanate from, and terminate at ODE-saddles $\bullet$.
The open cyclicity set $\mathbf{C}$ decomposes into five connected components: two annular cyclicity regions (orange, red), and three central cyclicity regions (blue), which are disks punctured by ODE-centers $\odot$.
The resulting regular parenthesis structure $(\,(\,(\,)\,(\,)\,)\,(\,)\,)$ indicates the traversal of the horizontal $\mathbf{v}$-axis by 2-periodic orbits at their nondegenerate extrema in the respective regions.
In the full lap signature $\mathfrak{S}$, the color coding of the lap lists $\underline{S}_k\,,\,\overline{S}_k$ corresponds to their cyclicity regions. 
See also \eqref{Sfiglap}.}}
\label{figlap}
\end{figure}

To any period map $T$ we may associate a specific lap signature $\fS(T)$, as illustrated in figure \ref{figlap}.
We mark spatially homogeneous ODE-saddles $\mathbf{e}$ by bullets $\bullet$, and ODE-centers by $\odot$.
All equilibria, and their lap numbers $\ell_j$\,, are ordered by their values at the left Neumann boundary $\mathbf{x}\!=\!0$.

Parentheses correspond to bounded saddle separatrices in figure \ref{figlap}. 
They indicate the nesting of cyclicity regions, i.e.\ of connected components of the cyclicity set $\mathbf{C}$ from \eqref{Tmap}, as follows.
By dissipativity \eqref{diss}, the cyclicity regions are bounded.
Their boundaries are given by bounded separatrices to ODE-saddles $\bullet$ and, in the case of disk-like central regions, by punctures at ODE-centers $\odot$.
Consider the intervals of definition $\mathbf{D}=\mathbf{C}\cap\{\mathbf{p}=0\}$ of the period map on the horizontal $\mathbf{v}$-axis; see \eqref{Tmap} again.
The min/max type of extrema of the periodic orbits, on the axis, never changes within any interval.
We attach parentheses to each separatrix crossings, according to the shape ``$($'' or ``$)$'' of the separatrix orbit, at that extremum.
Near ODE-saddles $\bullet$\,, we represent the nearby periodic orbits, instead, to satisfy definition \ref{defSign}(iii).
Near ODE-centers $\odot$ we do not place additional parentheses.
We omit parentheses of annular regions, for which the lists of 2-periodic orbits are empty.

Continuity of the period map $T$ in \eqref{Tmap}, and regularity of the values $T\!=\!2/\ell$ already impose strong restrictions on such lap signatures $\fS(T)$. 
These restrictions in fact imply that the lap signature $\fS(T)$ of $T\!=\!T_\mathbf{g}$ is a full lap signature $\fS$, in the sense of definition \ref{defSign}, for any $\mathbf{g}\!\in\! \SON$.
See figure \ref{figTg} again for an illustration, in particular of alternate jump condition (ix).
The full lap signature of the example in figure \ref{figlap} reads
\begin{equation}
\label{Sfiglap}
\fS=\bullet\,\Big{(}\underline{S}_1\left(\underline{S}_2(\underline{S}_3\odot\overline{S}_3)\bullet(\underline{S}_4\odot\overline{S}_4)\overline{S}_2\right)\bullet(\underline{S}_5\odot\overline{S}_5)\,\overline{S}_1\Big{)}\,\bullet\,,
\end{equation}
with appropriate lists $\underline{S}_k\,,\,\overline{S}_k$ of lap numbers.

Conversely, and much less obviously, the lap signatures $\fS(T_\mathbf{g})$ of such period maps realize all abstract full lap signatures $\fS$, for suitable spatially reversible $\mathbf{g}\!\in\!\SON$.
For a proof see proposition 3 in \cite{roc07}.

For later reference, we briefly recall how to read any full lap signature $\fS$ of definition \ref{defSign} as a plane forest.
For general background on trees and parentheses see for example \cite{catalan}.

\begin{definition}\label{deftree}
Let $\fS$ be a full lap signature.
Then the parenthesis structure of definition \ref{defSign}(iii) alone defines a plane forest which we call the \emph{signature forest} of $\fS$.
\emph{Vertices} are the pairs of matching parentheses.
They are labeled by the left-to-right ordering of $\fS$.
The \emph{child relation} is defined by nesting of parenthesis pairs.
\emph{Leaves}, i.e.\ vertices without children, are the \emph{vertices of central pairs}.
All remaining \emph{annular vertices}, i.e. the vertices of annular pairs, possess at least two children.
Each connected component of the signature forest is a plane (i.e. ordered) tree.
A connected signature forest is also called a \emph{signature tree}.
\end{definition}
In this definition of vertices, we can equivalently replace the matched pairs of parentheses by the pairs $\underline{S}_k\,,\ \overline{S}_k$ of lap lists (which may be empty at leaves, alias central vertices).

The main result of \cite{roc25} is that global attractors, and hence connection graphs, of nonlinearities $\mathbf{f},\mathbf{g}\!\in\!\SON$ with the same lap signature are $C^0$ orbit equivalent:
\begin{equation} 
\label{SignA}
\fS(T_\mathbf{f})=\fS(T_\mathbf{g}) \quad \Rightarrow \quad \cA_\mathbf{f}^\cN\cong\cA_\mathbf{g}^\cN \,. 
\end{equation} 

The proof is based on a 1-1 correspondence 
\begin{equation}
\label{Signsigma}
\fS(T_\mathbf{f})=\fS(T_\mathbf{g}) \quad \Leftrightarrow \quad \sigma_\mathbf{f}=\sigma_\mathbf{g}
\end{equation}
between lap signatures $\fS$ and those Sturm permutations $\sigma$ which arise from the reversible class $\mathbf{f},\mathbf{g}\in\SON$\,.
In \cite{firowo12a, roc25}, this subclass of all Sturm permutations has been characterized as \emph{integrable Sturm involutions}.
Indeed, we have already observed in \eqref{urev}, \eqref{sigmarev}, how reversibility \eqref{grev} of $\mathbf{g}$ implies the involution property $\sigma\!=\!\sigma^{-1}$ for $\sigma\!=\!\sigma_\mathbf{g}$.

The deeper notion of \emph{integrability} arises from the non-homogeneous min/max pairs $\underline{\mathbf{v}}_j\,,\ \overline{\mathbf{v}}_j\in\cEgN$ of \eqref{vpair}.
They come in two flavors, depending on the parity of their \emph{lap number} $\ell(v)=\mathbf{z}(\mathbf{v_x})$ in \eqref{Tell}. 
Each min/max pair with \emph{odd lap number} $\ell$ contributes a \emph{2-cycle} 
\begin{equation}
\label{2cycle}
(\underline c_{\,\alpha} \ \overline c_\alpha)\,
\end{equation}
to the Sturm permutation $\sigma\!=\!\sigma_\mathbf{g}$\,.
Without loss, we present the 2-cycle as $\underline c_{\,\alpha} < \overline c_\alpha$\,.
Lap numbers $\ell$ of \emph{even parity} only contribute fixed points.
In particular, again, $\sigma$ is an involution, given by
the disjoint 2-cycles \eqref{2cycle} which arise from odd lap numbers $\ell$.

The following terminology is natural, since 2-cycles \eqref{2cycle} are associated with periodic ODE orbits \eqref{gvp}, \eqref{Tell}:
\begin{itemize}
\item two distinct 2-cycles $(\underline c_{\,\alpha} \ \overline c_\alpha)$ and 
$(\underline c_{\,\beta} \ \overline c_\beta)$ are called \emph{intersecting}, if the corresponding intervals have a 
nonempty intersection, i.e. 
$(\underline c_{\,\alpha}, \overline c_\alpha) \cap (\underline c_{\,\beta}, \overline c_\beta) \ne \emptyset$\,;
\item intersecting 2-cycles $(\underline c_{\,\alpha} \ \overline c_\alpha)$ and 
$(\underline c_{\,\beta} \ \overline c_\beta)$ are called \emph{nested} if one of the intervals strictly contains 
the other, i.e. $(\underline c_{\,\beta} - \underline c_{\,\alpha})(\overline c_\alpha - \overline c_\beta)>0$\,;
\item nested 2-cycles $(\underline c_{\,\alpha} \ \overline c_\alpha)$ and 
$(\underline c_{\,\beta} \ \overline c_\beta)$ are \emph{centered} if the mid-points of the respective 
intervals coincide, i.e. 
$\underline c_{\,\beta} - \underline c_{\,\alpha} = \overline c_\alpha - \overline c_\beta$\,.
\end{itemize}
A fixed point $e$ of $\sigma$ is called $\sigma$\emph{-stable}, if it possesses formal Morse index $i_e\!=\!0$.
Similarly, we define the $\sigma$\emph{-stable core} $C_\alpha$ of a 2-cycle $(\underline c_{\,\alpha} \ \overline c_\alpha)$ to consist of its interior $\sigma$-stable points. 
We say that two nested 2-cycles of $\sigma$ are \emph{core-equivalent} if both share the same $\sigma$-stable core. 
The above terminology cumulates in the following definition \cite{firowo12a,roc25}.

\begin{definition}\label{defint}
An involution $\sigma$ is called \emph{integrable} if the following three properties hold:
\begin{enumerate}[(i)]
\item any two intersecting 2-cycles are nested;
\item any two core-equivalent 2-cycles are centered; and
\item any pair of non-nested 2-cycles is separated by a $\sigma$-stable point.
\end{enumerate}
\end{definition}

The necessary conditions collected in this definition characterize the additional restrictions of Sturm permutations $\sigma_\mathbf{g}^\cN$ in the spatially reversible class; see \cite{firowo12a, rofi14, roc25} for full details.

\begin{theorem}\label{thmHamrev}
A Sturm permutation $\sigma\!\in\! S_{2n+1}$ is realizable as the Sturm permutation $\sigma\!=\!\sigma_\mathbf{g}^\cN$ of some  spatially reversible nonlinearity $\mathbf{g}\!\in\!\SON$, if and only if $\sigma$ is an integrable involution.
The same statement holds true in the Hamiltonian class $\mathbf{g}\!\in\!\SH$.

In particular, any integrable Sturm involution $\sigma$ is realizable by a Hamiltonian nonlinearity $\mathbf{g}\!=\!\mathbf{g}(\mathbf{u})$, rather than just by some reversible $\mathbf{g}\!\in\!\SON$.
Moreover, the classes of Hamiltonian Sturm permutations and of integrable Sturm involutions coincide.
\end{theorem}

With these results we have reduced the question of counting the integrable Sturm involutions $\sigma\!=\!\sigma_\mathbf{g}^\cN$ resulting from reversible nonlinearities $\mathbf{g}\!\in\!\SON$, or just from Hamiltonian nonlinearities $\mathbf{g}\!\in\!\SHN$, to the enumeration of the full lap signatures $\fS$ of definition \ref{defSign}.
Technically, this is our approach which will prove the combinatorial results of section \ref{Intro}.
The ODE and PDE steps which then lead to associated lists of Sturm attractors have been outlined in section \ref{Back}, already.

Note how reducible Hamiltonian Sturm permutations, i.e. stacked associated Sturm global attractors, are equivalent to \emph{non-connected} signature forests, in the sense of definition \ref{deftree}. 
Indeed, the ODE-saddles $\bullet$ which separate root pairs of parentheses amount to gluing equilibria $\mathbf{v}_{N_*}$ of stacked attractors; see \eqref{stack}.
In contrast, $b_0(x,y)$ and its various associates listed just before theorem \ref{thmbasy}, all possess interpretations via generating functions associated to \emph{connected} signature forests, i.e.\ to signature trees, alias irreducible Hamiltonian Sturm permutations $\sigma$.

\section{The examples $n=r+q\leq 4$}\label{Ex}

\begin{table}[t!]
\begin{center}
\begin{tabular}{|| c || l | c | c | c | c ||} \hline\hline
     &        &             & & &  \\ [-3mm]
$n\!=\!r\!+\!q$ & $r.q$    & $b_{rq}$ & full lap signature $\fS$ & tree &  $\sigma\!\in\!\SHN$ 
\\[1mm] \hline\hline 
     &        &             & & &  \\ [-3mm]
1     & 1.0(i)          &       1       & $\bullet\,\big(\odot\big)\,\bullet$ & $\checkmark$ & $()\!\in\! S_3$ \\[1mm] 
\hline
     &        &             & & &  \\[-3mm]
2       & 1.1(i)          &       1       & $\bullet\,\big(\,\{1\}\odot\{1\}\,\big)\,\bullet$ & $\checkmark$ & $(2\,4)$ \\
     &        &             & & &  \\[-3mm]
     & 2.0(i)          &       1       & $\bullet\,\big(\odot\big)\bullet\big(\odot\big)\,\bullet$ & $-$ & $()\!\in\! S_5$
\\[1mm] 
\hline
     &        &             & & &  \\ [-3mm]
3       & 1.2(i)       &       2       & $\bullet\,\big(\,\{1,2\}\odot\{2,1\}\,\big)\,\bullet$ & $\checkmark$ & $(2\,6)$ \\ 
       & 1.2(ii)      &                & $\bullet\,\big(\,\{1,1\}\odot\{1,1\}\,\big)\,\bullet$ & $\checkmark$ & $(2\ 6)\,(3\,5)$ \\
     &        &             & & &  \\[-3mm]
       & 2.1(i)          &       2       & $\bullet\,\big(\,\{1\}\odot\{1\}\,\big)\bullet\big(\odot\big)\,\bullet$ & $-$ & $(2\,4)$ \\
     &        &             & & &  \\[-3mm] 
    & 3.0(i)          &       1       & $\bullet\,\big(\odot\big)\,\bullet\,\big(\odot\big)\,\bullet\,\big(\odot\big)\,\bullet$ & $-$ & $()\!\in\! S_7$ 
\\[1mm] 
\hline
     &        &             & & &  \\[-3mm]
4       & 1.3(i)       &       3       & $\bullet\,\big(\,\{1,2,3\}\odot\{3,2,1\}\,\big)\,\bullet$ & $\checkmark$ & $(2\,8)\,(4\,6)$ \\ 
       & 1.3(ii)      &                & $\bullet\,\big(\,\{1,2,2\}\odot\{2,2,1\}\,\big)\,\bullet$ & $\checkmark$ & $(2\,8)$ \\ 
       & 1.3(iii)     &                & $\bullet\,\big(\,\{1,1,1\}\odot\{1,1,1\}\,\big)\,\bullet$ & $\checkmark$ & $(2\,8)\,(3\,7)\,(4\,6)$ \\
     &        &             & & &  \\ [-3mm]
       & 2.2(i)       &       6       & $\bullet\,\big(\,\{1,2\}\odot\{2,1\}\,\big)\bullet\big(\odot\big)\,\bullet$ & $-$ & $(2\,6)$ \\
       & 2.2(ii)      &                & $\bullet\,\big(\,\{1,1\}\odot\{1,1\}\,\big)\bullet\big(\odot\big)\,\bullet$ & $-$ & $(2\ 6)\,(3\,5)$ \\
       & 2.2(iii)     &                & $\bullet\,\big(\,\{1\}\odot\{1\}\,\big)\bullet\big(\,\{1\}\odot\{1\}\,\big)\,\bullet$ & $-$ & $(2\ 4)\,(6\,8)$ \\                  
       & 2.2(iv)     &                & $\bullet\,\big(\,\{1,1\}\,\big(\odot\big) \bullet \big(\odot\big)\,\{1,1\}\,\big)\,\bullet$ & $\checkmark$ & $(2\,8)\,(3\,7)$  \\  
     &        &             & & &  \\[-3mm]     
       & 3.1(i)       &       3       & $\bullet\,\big(\,\{1\}\odot\{1\}\,\big)\,\bullet\,\big(\odot\big)\,\bullet\,\big(\odot\big)\,\bullet$ & $-$ & $(2\,4)$ \\
       & 3.1(ii)      &                & $\bullet\,\big(\odot\big)\,\bullet\,\big(\,\{1\}\odot\{1\}\,\big)\,\bullet\,\big(\odot\big)\,\bullet$ & $-$ & $(4\,6)$ \\
     &        &             & & &  \\[-3mm]   
     & 4.0(i)          &       1       & $\bullet\,\big(\odot\big)\,\bullet\,\big(\odot\big)\,\bullet\,\big(\odot\big)\,\bullet\,\big(\odot\big)\,\bullet$ & $-$ & $()\!\in\! S_9$ \\[-3mm]
     &        &             & & &  \\        
\hline\hline
\end{tabular} 
\end{center} 
\caption[Full lap signatures for $r+q\leq4$]{\small\emph{Table of the 21 full lap signatures for Hamiltonian nonlinearities $\mathbf{g}\in\SHN$ or, equivalently, for $\mathbf{g}\in\SON$.
The first two columns indicate $n\!=\!r\!+\!q\!\leq\!4$ and $r.q$\,.
Here $r$ counts the ODE-centers $\odot$ and $q$ counts the periodic ODE orbits of minimal periods $2/\ell_j$\,, alias the total number of lap entries $\ell_j$ in the lap lists $\underline{S}_k$\,.
In column four, the entries $b_{rq}$ of \eqref{bxy} count the number of full lap signatures $\fS$, including (omitted) reflections of $\fS$ under the spatial involution $\kappa$ from \eqref{sigmakappa}.
Omitting reflections reduces the case count from 21 to 17.
Checkmarks $\checkmark$ in the column ``tree'' highlight tree signatures $\fS$, alias irreducible Hamiltonian Sturm permutations $\sigma$, according to definition \ref{deftree} and \eqref{stack}.
Non-connected forests, alias reducible Sturm permutations or stacked global attractors are marked by $-$\,.
The last column lists the associated Hamiltonian Sturm permutations $\sigma\in S_{2n+1}$\,, in 2-cycle notation \eqref{2cycle}.
See also \cite{firowo12b}.
}}
\label{tablap}
\end{table}

Consider nonlinearities $\mathbf{g}\!\in\!\SON$ of spatially reversible type, or just the Hamiltonian case $\mathbf{g}\!\in\!\SHN$.
We enumerate the 21 full lap signatures $\fS$ with $r\!\geq\!1$ ODE-centers $\odot$, with $q$ orbits of (not necessarily minimal) period 2, and with $n\!=\!r\!+\!q\!\leq\!4$.
By the 1-1 correspondence \eqref{Signsigma}, equivalently, the signatures $\fS$ enumerate the corresponding Hamiltonian Sturm permutations, alias the integrable Sturm involutions $\sigma$.
See table \ref{tablap}.

In \cite{firowo12b} we had already listed all 17 cases with $N:=n+2q\leq9$.
The cases were distilled by a brute force scan of the 42 nontrivial Sturm permutations, known from \cite{fiedlertatra,firo96}, for integrable Sturm involutions.

At present, expansion \eqref{bxy} of the explicit generating function $b(x,y)$ in theorem \ref{thmbxy} provides the counts $b_{rq}$ of cases $r.q$ in table \ref{tablap}.
Reflections of $\fS$ under the spatial involution $\kappa$ of \eqref{sigmakappa}, however, have been omitted in the table.
The four asymmetric lap signatures 2.1(i), 2.2(i), 2.2(ii), and 3.1(i) therefore contribute doubly to $b_{rq}$.

By hand, it is actually easy to enumerate those signature forests $\fS$ in table \ref{tablap} which only consist of roots, i.e.\ of stacked central regions $\ldots(\underline{S}_k \odot\overline{S}_k)\ldots$ in $\fS$.
See \eqref{stack}, definitions \ref{defSign}(iv) and \ref{deftree}, and the last paragraph of section \ref{Sign}.
The intermediate ODE-saddles $\bullet$ simply concatenate such regions, by property (iii).
It then remains to fill up the lists $\underline{S}_k\,\ \overline{S}_k$ by a total of $2q$ entries, according to rules (vi)--(ix).

By rules (v), (vii), only case 2.2(iv) turns out to possess an annular pair $\ldots(\underline{S}_k (\ldots)\,\bullet\ldots\bullet\,(\ldots)\overline{S}_k)\ldots$ in $\fS$, with non\-empty lists $\underline{S}_k\,,\,\overline{S}_k$\,.
Indeed, even length $s_k\geq2$ of each annular list requires $q\geq2$.
Hence our constraint $r+q\leq4$ implies $1\leq r\leq2$.
Since $r\!=\!1$ does not admit interior ODE-saddles $\bullet$, we conclude $r\!=\!q\!=\!2$, as claimed.
Since a single annular list of minimal length $s_k\!=\!2\!=\!q$ uses up all lap number resources, \eqref{laplistbc} implies $\underline{S}\!=\!\overline{S}\!=\!\{1,1\}$\,.
This identifies the only annular case 2.2(iv).

For explicit sketches of the meanders $\cM$ and the Neumann connection graphs $\cC^\cN$ associated to the integrable Sturm involutions $\sigma$ of table \ref{tablap}, we refer to \cite{fiedlertatra,firo96,firowo12b}. 
For the six central cases $r.q\!=\!1.4$ see also cases 3.4(i)--(vi) of table 5.1 in \cite{rofiln}.
The cases $1.q(i)$ with $r\!=\!1$ and $\underline{S}=\{1,\ldots,q\}$ refer to the $(q\!+\!1)$-dimensional \emph{Chafee-Infante attractors} of $\mathbf{g}(u)=\lambda^2 u(1-u^2)$ with parameter $q\!<\!\lambda/\pi\!<\!q\!+\!1$.
See figure \ref{figCI3} for the case $q\!=\!2$, and \cite{firo22,rofiln} for detailed discussions of the general case, its history, and its peculiar characterizations.

\section{Proof of theorem \ref{thmbxy} on generating functions}\label{Pfb}

We recall that the coefficients $b_{rq}$ of the refined generating function $b(x,y)$ in \eqref{bxy} count the full lap signatures $\fS$.
Here $r\geq1$ denotes the number of centers $\odot$, and $q$ counts the number of lap entries in the annular or circular lap lists $\underline{S}_k$\,.

We prove the explicit expressions \eqref{bb0} and \eqref{b0} for $b(x,y)$ in theorem \ref{thmbxy} in three steps.

In our first step, let $b_0(x,y)$ denote the generating function of signature trees, according to definition \ref{deftree}.
Then the $m$-th power $b_0^m$ is the generating function counting signature forests which consist of $m$ sequentially labeled connected components.
Summation of the geometric series over $m\geq1$ yields $b=b_0/(1-b_0)$, proving claim \eqref{bb0}.

To prove the remaining claim \eqref{b0} on signature trees, we introduce the two generating functions
\begin{align}
\label{fa}
    \fa(y)&=\sum_{s\geq2\textrm{ even}}\fa_s\,y^s\quad\textrm{and}   \\
\label{fc}
    \fc(y)&=\quad\sum_{s\geq0}\ \ \fc_s\,y^s      
\end{align}
which count (nonempty) annular lists and (possibly empty) central lists $\underline{S}_k$ of length $s$, respectively.
See definition \ref{defSign}(iv)--(ix).

Two cases arise.
First, suppose that the root vertex of our signature tree $\fS$ is already central, i.e.\ a leaf.
This case contributes a term $x\fc(y)$ to the generating function $b_0$\,.

Second suppose the root vertex is annular.
Depending on its number $m\geq2$ of children, this contributes a recursion term $a(y)b_0^m$ to the generating function $b_0$\,.

Summation over all contributions from both cases provides a quadratic recursion
\begin{equation}
\label{b0rec}
b_0=x\fc(y)+\fa(y)\,b_0^2/(1-b_0)\,.
\end{equation}
for $b_0$\,.
We solve for $b_0$ in terms of $\fa(y)$ and $\fc(y)$, explicitly:
\begin{equation}
\label{b0ac}
b_0=\frac{1}{2(\fa + 1)}\left(1+x\fc-\sqrt{(1-x\fc)^2-4\fa\,x\fc}\right)\,.
\end{equation}
Here we have determined the sign of the square root such that $b_0(0,y)=0$, to accommodate the requirement $r\geq1$.
In the remaining two steps of lemma \ref{lemfa} and \ref{lemfc} we derive explicit expressions \eqref{faex} and \eqref{fcex} for the generating functions $\fa(y)$ and $\fc(y)$ of central and annular lap lists, respectively.
Inserting these expressions into \eqref{b0ac} then proves the remaining claim \eqref{b0} of theorem \ref{thmbxy}.

\begin{lemma}\label{lemfa}
The generating function $\fa(y)$, which counts annular lap lists, is
\begin{equation}
\label{faex}
\fa(y)=-1+\frac{1}{2y^2}\left(1-\sqrt{1-4y^2}\right)\,;\qquad \fa(0)=0\,.
\end{equation}
In particular, $\fa(y)$ is locally complex analytic for $|y|<1/2$.
The explicit coefficients $\fa_s$ of $y^s,\,s\!\geq\!0$, vanish for $s\!=\!0$ and for odd $s$.
For even $s\!=\!2s'\!\geq2\!$ they are the \emph{Catalan numbers}
\begin{equation}
\label{fas}
\fa_{s}=\mathfrak{C}_{s'}:=\frac{1}{s'+1}\binom{2s'}{s'}\,.
\end{equation}
\end{lemma}
\textbf{Proof.}\quad
The Catalan numbers $\mathfrak{C}_{s'}$ are listed under  \url{https://oeis.org/A000108}\,.
See also \cite{catalan}.
Among many equivalent variants, they count \emph{ballot sequences} $\boldsymbol{\beta}=\beta_1\ldots\beta_{2s}$ such that $\beta_j\!=\!\pm1$ possess nonnegative partial sums $\sum_{j\leq j_0}\beta_j$, and their total sum is zero.

We claim that ballot sequences $\boldsymbol{\beta}$ of length $2s\geq2$ are in 1-1 correspondence to annular lap lists $\underline{S}=\{\ell_1,\ldots,\ell_{2s}\}$.
To prove this, we recall the equivalence \eqref{Signsigma} and theorem \ref{thmHamrev} between full lap signatures and integrable Sturm involutions, alias the period maps $T\!=\!T_\mathbf{g}$ of PDE nonlinearities $\mathbf{g}\!\in\!\SHN$.
More precisely, let $\mathbf{v}_j$ denote the equilibrium associated to the lap number $\ell_j=\mathbf{z}(\mathbf{v}_{j,x})$ in $\underline{S}$.
Then the period map $v_0\mapsto T(v_0)$ crosses the regular level $T\!=\!2/\ell_j$ transversely, at $v_0=\mathbf{v}_j(0)$.
See figure \ref{figTg}(right) for an annular example.
Let $T_j':=T'(\mathbf{v}_j(0))\neq0$ denote the derivative of the time map, at the crossing, and
define the ballot $\boldsymbol{\beta}$ of $\underline{S}$ as
\begin{equation}
\label{betaT}
\beta_j:=-\,\mathrm{sign}\,T_j'\,.
\end{equation}
Indeed, the partial sums of upward crossings $\beta_j\!=\!-1$ of $T(v_0)$, alias ``$y$'' or ``$)$'',  can never exceed the accumulated downward crossings $\beta_j\!=\!+1$, alias ``$x$'' or ``$($'', strictly.
For annular $\underline{S}$, the last crossing of $T$ at $\ell_{2s}\!=\!1$ must be upward, with $\beta_{2s}\!=\!-1$.
Therefore each level $T\!=\!2/\ell_j$ must be crossed downwards as many times as upwards, and the total some of $\beta_j$ is zero.
Conversely, it is easy to see how any ballot $\boldsymbol{\beta}$ defines a unique annular lap list $\underline{S}$ which satisfies definition \ref{defSign}.

We may express this 1-1 correspondence by the explicit recurrence
\begin{equation}
\label{ellbeta}
\ell_{j+1}-\ell_j=\tfrac{1}{2}\left(\beta_j+\beta_{j+1}\right) \in \{-1,0,+1\},
\end{equation}
for $1\leq j<2s$, with the initializing conventions $\ell_1\!=\!1\!=\!\beta_1$\,.
Resorting to period maps circumvents the tedium of pedantic definition checking, here.

Comparing nonzero counts, for $s\geq1$, the 1-1 correspondence \eqref{betaT} alias \eqref{ellbeta} proves
\begin{equation}
\label{faC}
\fa_{2s'}=\mathfrak{C}_{s'}\,,
\end{equation}
and claim \eqref{fas} of the lemma.

By definition \ref{defSign}(v) for annular lap lists, $\fa(y)$ had to omit the empty term $\fa(0)\!=\!0$.
By \eqref{faC}, however, reinsertion of that term restores the known generating function 
\begin{equation}
\label{catx}
\mathfrak{C}(x)=\sum_{s'\geq0} \mathfrak{C}_{s'}\, x^{s'}=\frac{1}{2x}\left(1-\sqrt{1-4x}\right)
\end{equation}
of the Catalan numbers \cite{catalan}.
In other words, \eqref{faC} implies
\begin{equation}
\label{faCy}
\fa(y)=\mathfrak{C}(y^2)-1\,.
\end{equation}
Substitution of \eqref{catx} therefore proves the remaining claim \eqref{faex}, and the lemma.
\qed

\begin{lemma}\label{lemfc}
The generating function $\fc$, which counts central lap lists, is
\begin{equation}
\label{fcex}
\fc(y)=2\,/\left(1-2y+\sqrt{1-4y^2}\right)\,;\qquad \fc(0)=1\,.
\end{equation}
In particular, $\fc(y)$ is complex analytic for $|y|<1/2$.
The explicit coefficients $\fc_s$ of $y^s,\,s\!\geq\!0$, in terms of the floor function $\lfloor\cdot\rfloor$ are
\begin{equation}
\label{fcs}
\fc_s=\binom{s}{\lfloor s/2\rfloor}\,.
\end{equation}
\end{lemma}
\textbf{Proof.}\quad
We invoke the same 1-1 correspondence \eqref{betaT}, \eqref{ellbeta} between central lap lists $\underline{S}$ and ballot sequences  as in the proof of lemma \ref{lemfa}.
We drop the zero sum requirement for the totality of all $\beta_j$\,.
We therefore admit sequences of any length $s\geq0$, including the empty sequence.
We may symmetrize any such ballot by 
\begin{equation}
\label{beta2s}
\beta_{s+j}:=-\beta_{s+1-j}\,,\quad\textrm{for }1\leq j\leq s\,.
\end{equation}
This doubles the length of the ballot to $2s$, keeps partial sums nonnegative by symmetric repetition, and achieves zero total sum.

In other words, $\fc_s$ counts the resulting symmetric ballots $\boldsymbol{\beta}$ of length $2s\geq0$.
For their counts \eqref{fcs} see \url{https://oeis.org/A001405}.
The same reference provides the generating function
\begin{equation}
\label{fcex1}
\fc(y)=\frac{1}{2y}\left(-1+\sqrt{\frac{1+2y}{1-2y}}\right)\,,
\end{equation}
which coincides with claim \eqref{fcex}.
This proves the lemma.
\qed

As announced before, we can now insert the explicit expressions \eqref{faex} and \eqref{fcex}, for the generating functions $\fa(y)$ and $\fc(y)$ of central and annular lap lists, into the generating function \eqref{b0ac} of refined tree counts $b_0(x,y)$.
This proves the remaining claim \eqref{b0}, and completes the proof of refined counting theorem \ref{thmbxy}.

\section{Asymptotics of signature counts}\label{Asym}

In this section we derive the leading terms of asymptotic expansions for the coefficients of the generating functions  $\fa(y), \fc(y)$, and $a(z)$.
These refer to counts $\fa_{2s}, \fc_s$ annular and central lap lists $\underline{S}$ of lengths $2s$ and $s$, respectively, and to the counts $a_n$ of full lap signatures $\mathfrak{S}$ of type $r.q$ with prescribed $r+q\!=\!n$.
This section and the next closely follow methods developed and surveyed in the beautifully written book \cite{flajolet} by Flajolet and Sedgewick.

\begin{proposition}\label{propfac}
The asymptotics of the counts $\fa_{2s}$ and $\fc_s$ for $s\rightarrow\infty$ are
\begin{align}
\label{faas}
    \fa_{2s}=\mathfrak{C}_s&\sim\tfrac{1}{\sqrt{\pi}}\,4^s\,s^{-3/2} \,;  \\
\label{fcas}
    \fc_s&\sim\sqrt{\tfrac{2}{\pi}}\,2^s\,s^{-1/2}\,.
\end{align}
Here $\mathfrak{C}_s$ denotes the Catalan numbers.
The asymptotics $\sim$ indicate limiting quotients 1 of the left and right hand sides.
\end{proposition}

\textbf{Proof.}\quad Apply Stirling's formula to the binomial right hand sides of explicit expressions \eqref{fas}, \eqref{fcs}. 
\qed

The generating function $a(z)$ of corollary \ref{coraz} is slightly more involved. 
For example, we are not aware of any explicit expression for the counts $a_n$ of full lap signatures $\mathfrak{S}$ with a total of $n\!=\!r\!+\!q$ ODE-centers and ODE periodic orbits of minimal periods $2/\ell_j$\,.

We give a slightly adapted account of the procedure in \cite{flajolet}, which arrives at asymptotic expansions of a general generating function $\gamma(z)$ directly, without any explicit representations of its power series coefficients $\gamma_n$\,.
Assume a finite radius $0<\zeta<\infty$ of convergence for $\gamma(z)$.
Note $\zeta>0$ is real, because the coefficients of any generating function $\gamma(z)$ are real nonnegative.
In fact, $z\!=\!\zeta\!>\!0$ is the singularity of $\gamma$ closest to $z\!=\!0$.
More specifically, let us assume a decomposition of $\gamma(z)$ as
\begin{equation}
\label{gammazdec}
\gamma(z)=A(z)+B(z)\,\mathfrak{r}(z/\zeta)\,.
\end{equation}
Here $A(z), B(z)$ are assumed complex analytic in a larger disk $|z|\!<\!\rho\!>\!\zeta$, and the singularity $\mathfrak{r}$ abbreviates the standardized root
\begin{equation}
\label{root}
\mathfrak{r}(z):=(1-z)^{-\eta}
\end{equation}
with some complex $\eta\not\in\{0,-1,-2,\ldots\}$.
Then the much more general analysis in sections VI.1--VI.4 of \cite{flajolet} derives the following coefficient asymptotics:
\begin{equation}
\label{gammansim}
\gamma_n\,=\, \tilde{\gamma}_n\,(1+\cO(\tfrac{1}{n}))\,,\quad\textrm{with}\quad \tilde{\gamma}_n:=\frac{B(\zeta)}{\Gamma(\eta)}\;\zeta^{\,-n}\,n^{\eta-1}\,,
\end{equation}
for $n\rightarrow\infty$.
The proof in \cite{flajolet} is based on a careful analysis of decompositions like \eqref{gammazdec} via the residue theorem, along suitable complex contours of Hankel type.

To gain confidence, and routine, let us first illustrate the method \eqref{gammazdec}--\eqref{gammansim} and validate the asymptotic expansions \eqref{faas} and \eqref{fcas}, without the binomial Stirling detour \eqref{fas}, \eqref{fcs}.
For $\gamma$ we pick the generating functions $\mathfrak{a}(y)=\mathfrak{C}(y^2)-1$ in \eqref{faex}, \eqref{faCy}, and $\mathfrak{c}(y)$ in \eqref{fcex}. 
For both, the Catalan generating function $\mathfrak{C}(x)$ in \eqref{catx} and the generating function $\mathfrak{c}(y)$ in \eqref{fcex1}, we first multiply by the denominators $x$ and $y$, respectively, to shift the indices $s$ by 1.
We then obtain asymptotics of the coefficients analogous to \eqref{gammansim}, in either case.
Indeed the decomposition \eqref{gammazdec} for $x\,\mathfrak{C}(x)$ has $\eta=-1/2,\ \zeta\!=\!1/4,\  A\!=\!1/2,\ B\!=\!-1/2$, and $\Gamma(\eta)\!=\!-2\sqrt{\pi}$.

Similarly, $y\mathfrak{c}(y)$ leads to $\eta\!=\!+1/2,\ \zeta\!=\!1/2,\ A\!=\!-1,\ B\!=\!\tfrac{1}{2}\sqrt{1+2y}$, and $\Gamma(\eta)\!=\!\sqrt{\pi}$.
Invoking asymptotics \eqref{gammansim} proves proposition \ref{propfac}, once again.

We now return to the generating functions $\gamma(z)=a(z)=a_0/(1-a_0)$ and $\gamma(z)\!=\!a_0(z)$ of corollary \ref{coraz} proper.
Inserting subscripts $0$ to denote the tree counts $a_{0n}$ we will obtain analogous decompositions \eqref{gammazdec} and asymptotics \eqref{gammansim} with a leading coefficient $B_0(\zeta)$, instead.

\begin{theorem}\label{thmasa}
Let $z\!=\!\zeta>0$ denote the smallest positive zero of the term $\mathfrak{p}_*(z)$ of \eqref{fpz} under the square root in the numerator of $a_0(z)$ and $a(z)=a_0/(1-a_0)$\,; see \eqref{a0}.
Numerically,
\begin{equation}
\label{zeta}
    \zeta= 0.304\,293\,818\,497\,977\ldots\,.  
\end{equation}
Then $a_0(z)$ and $a(z)$ are complex analytic for $|z|<\zeta$, with a single singularity of type $\mathfrak{r}(z/\zeta)=\sqrt{1-z/\zeta}$ on the circle $|z|\!=\!\zeta$.
In particular, we obtain asymptotics \eqref{gammansim} with $\eta\!=\!-1/2,\ \Gamma(\eta)\!=\!-2\sqrt{\pi}$, and the universal leading coefficients
\begin{align}
\label{B0zeta}
-B_0(\zeta)/(2\sqrt{\pi})&= 0.178\,901\,287\,554\,763\ldots\quad\textrm{for }\ \gamma_n=a_{0n}\,,\textrm{ and}\\
\label{Bzeta}
\phantom{_0}-B(\zeta)/(2\sqrt{\pi})&= 1.732\,588\,240\,483\,502 \ldots\quad\textrm{for }\  \gamma_n=a_{n} \,.
\end{align}
\end{theorem}

\textbf{Proof.}
We start with the tree generating function $a_0(z)$ of \eqref{a0}.
To prepare a decomposition \eqref{gammazdec} for $a_0$, we first address the denominators $1-\sqrt{1-4z^2}$ and 
\begin{equation}
\label{a0denom}
1-2z+\sqrt{1-4z^2}=\sqrt{2(1-2z)(1+\sqrt{1-4z^2})}\,.
\end{equation}
This establishes analyticity for $|z|<1/2$.

Next, we decompose the numerator term $\mathfrak{p}_*(z)$ from \eqref{fpz} as
\begin{equation}
\label{fp1z}
\mathfrak{p}_*(z)=\mathfrak{p}_{*1}(z)\cdot(\zeta-z)\,.
\end{equation}
Here $\mathfrak{p}_*(\zeta)=0$ implies complex analyticity of $\mathfrak{p}_{*1}(z)$, for $|z|<1/2$. 
See \eqref{zeta} for the numerical value of $z\!=\!\zeta$.
For later reference, we also compute the unique real zero $\rho$ of $\mathfrak{p}_{*1}$\,, alias the second real zero of $\mathfrak{p}_*$\,, as
\begin{equation}
\label{rho}
z=\rho= 0.489\,855\,871\,726\,143\ldots\,. 
\end{equation}
The remaining two zeros are complex conjugate, of larger absolute value.
(We omit the explicit quartic expressions of all zeros by the Ferrari-Cardano formula.)
All four zeros are algebraically simple, of course.
For later reference, we also compute
\begin{equation}
\label{fp1zeta}
\mathfrak{p}_{*1}(\zeta)=-\mathfrak{p}_*'(\zeta)=4.615\,301\,063\,108\,272\ldots\neq0\,.
\end{equation}

We can now decompose $a_0=A_0+B_0\,\mathfrak{r}$ from \eqref{a0}, as in \eqref{gammazdec}, with exponent $\eta\!=\!-1/2$ in \eqref{root}, and
\begin{align}
\label{A0z}
A_0(z)&:=\phantom{-}\frac{z^2}{1-\sqrt{1-4 z^2}}\cdot
\left(1+\frac{2z}{1-2z+\sqrt{1-4z^2}}\right)\,;
\\[2mm]
\label{B0z}
B_0(z)&:= -\frac{\sqrt{2}\,z^2}{1-\sqrt{1-4 z^2}}\cdot
\frac
{\sqrt{\mathfrak{p}_{*1}(z)\,\zeta}}
{1-2z+\sqrt{1-4z^2}}\,.
\end{align}
Obviously, $A_0(z)$ is complex analytic for $|z|<1/2$.
For $B_0(z)$, we only obtain complex analyticity in the smaller disk $|z|<\rho<1/2$ delimited by \eqref{rho}.
Since $\rho>\zeta$, expansion \eqref{gammansim} applies with $\eta\!=\!-1/2$.
Evaluation at $z\!=\!\zeta$, with \eqref{fp1zeta} and $\Gamma(-1/2)\!=\!-2\sqrt{\pi}$, proves claim \eqref{B0zeta}.

To prove claim \eqref{Bzeta} for the full generating function $a(z)=a_0/(1-a_0)$, we first observe
\begin{equation}
\label{a0<1}
|a_0(z)|\leq a_0(\zeta)<1\,,\quad\textrm{for }|z|\leq\zeta\,.
\end{equation}
Indeed, $|a_0(z)|\leq a_0(|z|)\leq a_0(\zeta)$, since $a_0(z)$ is a generating function with nonnegative real coefficients.
Direct evaluation of \eqref{a0} and \eqref{zeta} proves \eqref{a0<1} with
\begin{equation}
\label{a0zeta}
a_0(\zeta)=0.678\,664\,245\,184\,786\ldots<1\,.
\end{equation}
By continuity of $a_0$, we may extend \eqref{a0<1} to a slightly larger disk $|z|<\varrho>\zeta$.

We can therefore insert decomposition \eqref{A0z}, \eqref{B0z} as
\begin{equation}
\label{aa0}
a=\frac{a_0}{1-a_0}=\frac{A_0+B_0\,\mathfrak{r}}{1-A_0-B_0\,\mathfrak{r}}=:A+B\,\mathfrak{r}\,.
\end{equation}
Here we expand the denominator as a geometric series in $\mathfrak{r}$, and subsume even powers of $\mathfrak{r}$ into $A$, and odd powers into $B\,\mathfrak{r}$.
Note analyticity of $A,B$ for $|z|<\varrho>\zeta$.
Therefore asymptotics \eqref{gammansim} apply, with $\eta\!=\!-1/2$.
Although the square root $\mathfrak{r}(z/\zeta)=\sqrt{1-z/\zeta}$ depends on $z$, we may evaluate $B(\zeta)$ by treating $\mathfrak{r}$ as an independent variable:
\begin{equation}
\label{Bzetar}
B(\zeta)=\partial_\mathfrak{r}\big{|}_{\mathfrak{r}=0}\,a=\partial_\mathfrak{r}\big{|}_{\mathfrak{r}=0}\,\frac{-1\phantom{0}\ }{1-1/a_0}=\frac{1}{(1-a_0)^2}\,\partial_\mathfrak{r}\big{|}_{\mathfrak{r}=0}\,a_0=\frac{B_0(\zeta)}{(1-a_0(\zeta))^2}\,.
\end{equation}
Substitution of \eqref{B0zeta} and \eqref{a0zeta} validates claim \eqref{Bzeta}, and proves the theorem. \qed

\section{Gaussian asymptotics of refined signature counts}\label{Gauss}

In the first part of this section we prove the Gaussian asymptotics of theorem \ref{thmGauss}.
The proof is an adaptation of theorems IX.8 and IX.12 in \cite{flajolet}. 
In the second part we derive the asymptotics of the Gaussian means $\boldsymbol{\mu}_n$ and variances $\boldsymbol{\sigma}^2$.
We reuse some notation already familiar from the previous section.
We formulate the requirements on general bivariate generating functions $f\!=\!f(u,z)$, from \cite{flajolet}, which will then be readily verified in our specific case of the refined generating function $f(u,z):=b(uz,z)$ from \eqref{uz}.

Let $f\!=\!f(u,z)$ be a general generating function, in $z$, with parameter $u$.
Reminiscent of \eqref{gammazdec}, we assume a decomposition of $f$ as
\begin{equation}
\label{fuzdec}
f(u,z)=A(u,z)+B(u,z)\,\left(\mathfrak{p}(u,z)\right)^{-\eta}\,,
\end{equation} 
with some fixed complex $\eta\not\in\{0,-1,-2,\ldots\}$.
Here $A,B,\mathfrak{p}$ are assumed to be complex analytic in a $\bC^2$ disk $|u-1|<\eps,\ |z|<\varrho$.
Similarly to $\mathfrak{p}_*(z)$ in \eqref{fpz}, we also assume that $\mathfrak{p}(1,z)=0$ possesses a unique root $z\!=\!\zeta$ with $|z|<\varrho$.
Moreover, let $\zeta\neq0$ and 
\begin{equation}
\label{neq01}
B\neq0\,,\ \partial_u \mathfrak{p}\neq0\,,\ \partial_z \mathfrak{p}\neq 0 
\end{equation}
at $(u,z)=(1,\zeta)$.
Reducing $\eps>0$, if necessary, the implicit function theorem then asserts a unique locally analytic continuation $\zeta(u)$ of the root $z\!=\!\zeta$ in the disk, such that
\begin{equation}
\label{zetau}
\mathfrak{p}(u,\zeta(u))=0\,.
\end{equation}
The final assumption is nondegeneracy of the \emph{analytic variance} $\mathfrak{v}$\,:
\begin{equation}
\label{var}
\begin{aligned}
    \mathfrak{v}(\zeta(1)/\zeta(u))&\phantom{:}\neq0\,,\quad\textrm{where}   \\
    \mathfrak{v}(\xi)&:=\xi''(1)+\xi'(1)-\xi'(1)^2\,,
\end{aligned}
\end{equation}
for any $\xi\!=\!\xi(u)$ such that $\xi(1)\!=\!1$.

In terms of the power series $f(u,z)=\sum_{nr}f_{nr}\,u^rz^n$, the normalization \eqref{pnu} always defines a sequence of random variables $X_n$ with probabilities $\bP(X_n\!=\!r):=p_{nr}$ and generating functions $p_n(u)=\sum_rp_{nr}u^r$.
Under the assumptions \eqref{fuzdec}--\eqref{var}, theorems IX.8 and IX.12 of \cite{flajolet} then imply the conclusions of theorem \ref{thmGauss}, for any general generating function $f\!:=\!f_0(u,z)$ and $f\!:=\!f(u,z)$.

\textbf{Proof of Theorem \ref{thmGauss}.}\quad 
It only remains to validate the above assumptions for our two specific cases $f(u,z):=f_0(u,z)\!=\!b_0(uz,z)$ defined in \eqref{f0} and $f(u,z):=b(uz,z)\!=\!f_0/(1-f_0)$, respectively.
We recall that the first case refers to refined counting of signature trees, whereas the second case counts signature forests; see definition \ref{deftree}.
In section \ref{Asym}, we have already studied the case $u\!=\!1$ of the summations $a_0(z)\!=\!f_0(1,z)$ and $a(z)\!=\!f(1,z)$.

For tree counting $f_0$, we obtain a decomposition \eqref{fuzdec} with $\eta\!=\!1/2$ and
\begin{align}
\label{A0uz}
A_0(u,z)&:=\phantom{-}\frac{z^2}{1-\sqrt{1-4 z^2}}\cdot
\left(1+\frac{2uz}{1-2z+\sqrt{1-4z^2}}\right)\,;
\\[2mm]
\label{B0uz}
B_0(z)&:= \frac{z^2}{1-\sqrt{1-4 z^2}}\cdot\frac{\sqrt{2}}{1-2z-\sqrt{1-4z^2}}\,;
\\[2mm]
\label{fpuz}
\mathfrak{p}(u,z)&:=\mathfrak{p}_*(z)+2 \left(2+(2-z)\sqrt{1-4 z^2}+3z\right) (u-1)-2 z^2 (u-1)^2\,.
\end{align}
Here $\mathfrak{p}(1,z)\!:=\!\mathfrak{p}_*(z)$ is inherited from \eqref{fpz}, of course.
Uniqueness and simplicity of the simple root $z\!=\!\zeta\!=\!\zeta(1)\!>\!0$ follows as in the proof of theorem \ref{thmasa},  \eqref{zeta}; see \eqref{fp1z}, \eqref{fp1zeta}.
Local analyticity of $f_0$ in $(u,z)$ implies analyticity of $A_0\,,B_0\,,\mathfrak{p}$ in the above disk, for some $\varrho>\zeta$ and sufficiently small $\eps>0$, such that uniqueness and analyticity of $\zeta(u)$ also holds.

To validate requirement \eqref{neq01}, first note $B_0(\zeta)\neq0$, by \eqref{B0uz}, \eqref{zeta}. 
We then calculate
\begin{equation}
\label{dufpuz}
\partial_u \mathfrak{p}(1,\zeta)=2 \left(2-(2-\zeta)\sqrt{1-4 \zeta^2}-3\zeta\right) = -0.516\,803\,397\,750\,558\ldots\neq0\,.
\end{equation}
Brute force calculation of the analytic variance $\mathfrak{v}$ in \eqref{var} for $\xi_1(u):=\zeta(1)/\zeta(u)$ confirms
\begin{equation}
\label{varnum}
\begin{aligned}
    \xi_1'(1)&=  \phantom{-}0.367\,986\,760\,847\,991 \ldots\,;\\
    \xi_1''(1)&=  -0.045\,774\,184\,178\,653\ldots\,;\\
    \mathfrak{v}(\xi_1(u))&=\phantom{-}\xi_1''(1)+\xi_1'(1)-\xi_1'(1)^2= 0.186\,798\,320\,509\,941\ldots\neq0\,.
\end{aligned}
\end{equation}
In particular, $\xi_1'(1)\neq0$ validates the nondegeneracy properties of $\mathfrak{p}(u,z)$ in assumption \eqref{neq01}, and $\mathfrak{v}(\xi_1(u))\neq0$ shows that the remaining assumption \eqref{var} is also satisfied.
This completes the proof of theorem \ref{thmGauss}, for the refined counts $f_0$ of signature trees, and for $Y_{0n}$ given by \eqref{Y0n}.

To address forest counting $f=f_0/(1-f_0)$, we proceed analogously to \eqref{aa0}, \eqref{Bzetar}, with $\mathfrak{r}:=\sqrt{\mathfrak{p}}$.
We obtain a decomposition \eqref{fuzdec} with
\begin{equation}
\label{Buz}
B(1,\zeta)=B_0(\zeta)/(1-a_0(\zeta))^2\neq0\,.
\end{equation}
Here $\zeta:=\zeta(1)$ is given by \eqref{zeta} and $a_0(\zeta)=f_0(1,\zeta)$ was already determined in \eqref{a0zeta}.
The exponent $-\eta\!=\!1/2$ and the singular term $\mathfrak{p}(u,z)$ remain unchanged.
Therefore, the proof of the remaining Gaussian claim \eqref{Yn} for the refined counts $f$ of all signature forests proceeds as before, and theorem \ref{thmGauss} is proved.
\qed

\textbf{Proof of Theorem \ref{thmbasy}.}\quad
Again, we start from the refined tree signature counts $f_0$\,, with coefficients $f_{0n}(u)$ defined analogously to \eqref{fnr}\,.
To determine the asymptotics of theorem \ref{thmbasy} on means $\boldsymbol{\mu}_{0n}$ and variances $\boldsymbol{\sigma}_{0n}^2$\,, for $n\rightarrow\infty$, we revisit the analytic approach already outlined in \eqref{gammazdec}--\eqref{gammansim} above.
However, the bivariate generating functions $f_0(u,z)=b_0(uz,z)$ of \eqref{uz}, \eqref{fnr} in $z$, and the sequence $p_{0n}(u)$ of normalized probability generating functions, now contain an additional parameter $u$ such that
\begin{align}
\label{mu0n}
    \boldsymbol{\mu}_{0n}&=p_{0n}'(1)\,;   \\
\label{sigma0n}
    \boldsymbol{\sigma}_{0n}^2&=p_{0n}''(1)+p_{0n}'(1)-p_{0n}'(1)^2\,;
\end{align}
compare also \eqref{var}.
In our previous decomposition method \eqref{gammazdec} for $\gamma:=f_0(u,z)$, the singularity term $\mathfrak{r}(z/\zeta)$ would now depend on $\zeta\!=\!\zeta(u)$, with a movable singularity location $z\!=\!\zeta(u)$ depending on $u$; see \eqref{zetau}.
To achieve uniformity in $u$, we therefore rescale $z$ and discuss the bivariate generating function
\begin{equation}
\label{g0uzdecB}
g_0(u,z):=f_0(u,\zeta(u)z)=\tilde{A}_0(u,\zeta(u)z)+\tilde{B}_0(u,\zeta(u)z)\,\mathfrak{r}\,,
\end{equation} 
with $\mathfrak{r}\!=\!\sqrt{1-z}$, instead.
This pins the movable singularity of $\mathfrak{r}$ to $z\!=\!1$, independently of the parameter $u$.
By analyticity of $\tilde{A}_0$ and $\tilde{B}_0\neq0$ in a complex disk $|z|<\varrho>1,\ |u-1|<\eps$, we may omit the term $\tilde{A}_0$ in the expansion
\begin{equation}
\label{g0uzdec}
g_0(u,z)=\tilde{B}_0(u,\zeta(u))\mathfrak{r}+\cO(\mathfrak{r}^3)\,,
\end{equation}
for $z\rightarrow1$ and uniformly in $u$.
Indeed, the contributions of $\tilde{A}_0$ are exponentially small compared to $\tilde{B}_0\,\mathfrak{r}$, for large $n$.
Therefore \eqref{gammansim} with $\eta\!=\!-1/2$ implies
\begin{align}
\label{f0nuas}
f_{0n}(u)&\phantom{:}=\,-\tfrac{1}{2\sqrt{\pi}}\,\psi_0(u)\,\zeta(u)^{-n}\,n^{-3/2}\cdot\left(1+\cO(\tfrac{1}{n})\right)\,,\quad\textrm{with}\\
\label{psi0a}
\psi_0(u)&:=\phantom{-}\tilde{B}_0(u,\zeta(u))\,.
\end{align}
The expansion holds uniformly in $u$ and in any finite number of $u$-derivatives.
Inserting $u\!=\!1$ we obtain
\begin{equation}
\label{a0nuas}
a_{0n}=f_{0n}(1)=\,-\tfrac{1}{2\sqrt{\pi}}\,\psi_0(1)\,\zeta(1)^{-n}\,n^{-3/2}\left(1+\cO(\tfrac{1}{n})\right)
\end{equation}
Taking quotients \eqref{f0nuas}/\eqref{a0nuas}, we find
\begin{equation}
\label{p0nuas}
p_{0n}(u)=\,\frac{\psi_0(u)}{\psi_0(1)}\,\left(\frac{\zeta(1)}{\zeta(u)}\right)^n+\cO(\tfrac{1}{n})\,.
\end{equation}
We can now apply theorem IX.8 of \cite{flajolet}, with $\beta_n\!=\!\kappa_n\!=\!n$ there, to arrive at the expansions \eqref{muas} and \eqref{varas} claimed in theorem \ref{thmbasy}.
It only remains to calculate the numerical values of the coefficients
\begin{align}
\label{m1}
\mathfrak{m}_1\  &= \ \, -\zeta'(1)/\zeta(1)\,;   \\ 
\label{m0}
\mathfrak{m}_0\, &= \, \phantom{-}\psi_0'(1)/\psi_0(1)\,;   \\
\label{v1}
\mathfrak{v}_1\  &= \ \, \mathfrak{v}(\zeta(1)/\zeta(u))\,;    \\ 
\label{v0}
\mathfrak{v}_0\, &= \ \mathfrak{v}(\psi_0(u)/\psi_0(1))\,.   
\end{align}
See \eqref{var} for the definition of the analytic variance functional $\mathfrak{v}(\cdot)$.

For $f_0$\,, we have already determined the numerical values of $\mathfrak{m}_1=\xi'(1)$ and $\mathfrak{v}_1=\mathfrak{v}(\xi_1(u))$ in \eqref{varnum}.
This proves the numerical claims \eqref{m1num} and \eqref{v1num}, albeit for tree counting $f_0$ instead of forest counting $f$.

To evaluate \eqref{m0} and \eqref{v0}, we use that the root $z\!=\!\zeta(u)$ of $z\mapsto\mathfrak{p}(u,z)$ remains simple; see \eqref{zetau}. 
Therefore
\begin{equation}
\label{fp1uz}
\begin{aligned}
    \mathfrak{p}(u,z)&=\phantom{-}\mathfrak{p}_1(u,z)(\zeta(u)-z)\,,\quad\textrm{with}   \\
    \mathfrak{p}_1(u,\zeta(u))&=-\partial_z\vert_{z=\zeta(u)}\,\mathfrak{p}(u,z)\,.
\end{aligned}
\end{equation}
Analyticity of $\mathfrak{p}_1$ holds in the usual domain $|u-1|<\eps,\ |z|<\varrho>\zeta$.
This identifies
\begin{equation}
\label{psi0p}
\psi_0(u)=\tilde{B}_0(u,\zeta(u))=B_0(\zeta(u))\sqrt{\mathfrak{p}_1(u,\zeta(u))\cdot\zeta(u)}\,,
\end{equation}
with $B_0$ from \eqref{B0uz} and $\mathfrak{p}_1$ from  \eqref{fp1uz}.
Explicit evaluation by implicit differentiation of $\mathfrak{p}(u,\zeta(u))=0$ provides the numerical values
\begin{align}
\label{psi00num}
\psi_0(1) &= -0.634\,188\,552\,116\,790\ldots\,;   \\
\label{psi01num}
\psi_0'(1) &= \phantom{-}0.167\,957\,887\,872\,550 \ldots\,;   \\
\label{psi02num}
\psi_0''(1) &= -0.187\,270\,666\,368\,600\ldots\,.
\end{align}
These imply the numerical values \eqref{m00num}, \eqref{v00num} of the constant coefficients $\mathfrak{m}_{0,0}\,,\ \mathfrak{v}_{0,0}$\,, again for tree counting $f_0$ instead of forest counting $f$.

It remains to address the counts $f=f_0/(1-f_0)$, alias $g=g_0/(1-g_0)$ of all signature forests.
Omitting all indices 0, which referred to $f_0$, we then argue as in \eqref{mu0n}--\eqref{p0nuas}.
This proves the asymptotics
\begin{align}
\label{pnuas}
p_n(u)&\phantom{:}=\,\frac{\psi(u)}{\psi(1)}\,\left(\frac{\zeta(1)}{\zeta(u)}\right)^n+\cO(\tfrac{1}{n})\,,\quad\textrm{where}\\[1mm]
\label{psi}
\psi(u)&:=\,\tilde{B}(u,\zeta(u))
\end{align}
originates from the decomposition 
\begin{equation}
\label{guzdecB}
g(u,z):=f(u,\zeta(u)z)=\tilde{A}(u,\zeta(u)z)+\tilde{B}(u,\zeta(u)z)\,\mathfrak{r}\,.
\end{equation} 
This is just \eqref{g0uzdecB}, with indices 0 removed and $\mathfrak{r}\!=\!\sqrt{1-z}$ again.

To determine $\tilde{B}(u,\zeta(u))$ in \eqref{psi}, \eqref{guzdecB}, we first replace $a$ and $a_0$ in \eqref{aa0}, \eqref{Bzetar} by the rescaled versions $g$ and $g_0$ of $f$ and $f_0$, respectively. 
Thus decomposition \eqref{g0uzdecB} implies
\begin{equation}
\label{Buzetar}
\psi(u)=\tilde{B}(u,\zeta(u))=\frac{\tilde{B}_0(u,\zeta(u))}{(1-A_0(u,\zeta(u)))^2}\,,
\end{equation}
with $\zeta(u)$ determined by \eqref{zetau}, and $A_0(u,\zeta(u)),\ \tilde{B}_0(u,\zeta(u))$ given by \eqref{A0uz}, \eqref{psi0p}.
By analyticity of the denominator, this establishes expansion \eqref{pnuas} with the same definition \eqref{zetau} of $\zeta(u)$ as before.
In particular, the leading coefficients $\frak{m}_1=\frak{m}_{0,1}$ and $\frak{v}_1=\frak{v}_{0,1}$ remain unchanged, with numerical values \eqref{m1num} and \eqref{v1num}, respectively.

Explicit evaluation of $\psi$, by implicit differentiation of \eqref{Buzetar} at $\mathfrak{p}(u,\zeta(u))=0$, also provides the numerical values
\begin{align}
\label{psi0num}
\psi(1) &= -6.141\,865\,397\,757\,192 \ldots\,;   \\
\label{psi1num}
\psi'(1) &= -3.481\,866\,148\,150\,836 \ldots\,;   \\
\label{psi2num}
\psi''(1) &= -0.224\,866\,907\,348\,425 \ldots\,.
\end{align}
These imply the numerical values \eqref{m0num}, \eqref{v0num} of the constant coefficients $\mathfrak{m}_0\,,\ \mathfrak{v}_0$\,, for the refined counting $f$ of signature forests.
Note how these values differ from tree counting in \eqref{m00num}, \eqref{v00num}.
This completes the proof of theorem \ref{thmbasy}.
\qed

\section{Numerical illustrations}\label{Num}

We illustrate our main results with a few numerical examples.
Up to $n=10\,000$, we explore the total counts $a_n$ of full lap signatures $\fS$, alias Hamiltonian Sturm permutations $\sigma$, with $N\!=\!2n\!+\!1$ Neumann equilibria.
See definitions \ref{defSign}, \ref{defint}, and theorems \ref{thmHamrev}, \ref{thmasa}.
For $n=4,13,21,30$, we then explore the Gaussian asymptotics of the normalized refined counts $b_{rq}/a_n$.
The indices $r$ and $q$ refer to the signatures, alias Sturm permutations, with $r$ ODE-centers $\odot$ and $q\!=\!n-r$ periodic orbits of minimal period $2/\ell$, for some lap numbers $\ell=1,2,3,\ldots$\,. 
See \eqref{b/a} and theorems \ref{thmGauss}, \ref{thmbasy}.
All results are based on the generating functions $a,a_0$ and $f,f_0$, respectively.
Only for $n\leq11$ have we verified our present counts against much older lists of all Sturm permutations, which were generated by the methods developed in \cite{firo96,firo99,firo00} and outlined in section \ref{Back}.

\begin{table}[p]
\centering
\begin{tabular}{|| r || r | r | c ||} \hline\hline
 &     &    &          \\[-3mm]
\multicolumn{1}{||c||}{$n$} & \multicolumn{1}{c |}{$a_{0n}$}    & \multicolumn{1}{c |}{$\sim\,\tilde{a}_{0n}$} & $n\cdot\cO(1/n)$  \\[1mm] \hline\hline 
 &     &    &          \\[-3mm]
 $4 $&$ 4 $&$ 2.6083 $&$ -2.1344 $\\
 $13 $&$ 21\,767 $&$ 19902. $&$ -1.2185 $\\
 $21 $&$ 1.3793\times 10^8 $&$ 1.3186\times 10^8 $&$ 0.96534 $\\
 $30 $&$ 3.5585\times 10^{12} $&$ 3.4525\times 10^{12} $&$ 0.92053 $\\
 $1\,000 $&$ 2.8831\times 10^{511} $&$ 2.8806\times 10^{511} $&$ 0.85279 $\\
 $2\,000 $&$ 5.1879\times 10^{1027} $&$ 5.1857\times 10^{1027} $&$ 0.85188 $\\
 $3\,000 $&$ 1.4377\times 10^{1544} $&$ 1.4373\times 10^{1544} $&$ 0.85158 $\\
 $4\,000 $&$ 4.7544\times 10^{2060} $&$ 4.7534\times 10^{2060} $&$ 0.85143 $\\
 $5\,000 $&$ 1.7321\times 10^{2577} $&$ 1.7318\times 10^{2577} $&$ 0.85134 $\\
 $6\,000 $&$ 6.7092\times 10^{3093} $&$ 6.7082\times 10^{3093} $&$ 0.85128 $\\
 $7\,000 $&$ 2.7109\times 10^{3610} $&$ 2.7105\times 10^{3610} $&$ 0.85123 $\\
 $8\,000 $&$ 1.1298\times 10^{4127} $&$ 1.1296\times 10^{4127} $&$ 0.85120 $\\
 $9\,000 $&$ 4.8208\times 10^{4643} $&$ 4.8204\times 10^{4643} $&$ 0.85118 $\\
 $10\,000 $&$ 2.0958\times 10^{5160} $&$ 2.0956\times 10^{5160} $&$ 0.85116 $\\ [-3mm]
 &     &    &         \\           
\hline\hline
\end{tabular} 
\caption[Asymptotic counts $a_{0n}$ of signature trees]{\small\emph{
Total counts $a_{0n}$ of signature trees $\fS$, alias irreducible Hamiltonian Sturm permutations $\sigma$, according to definition \ref{deftree} and \eqref{stack}.
We list counts for fourteen select $n$ from 4 to 10\,000.
Column three lists the asymptotics $a_{0n}\!\sim\!\tilde{a}_{0n}$ of \eqref{gammansim}, with leading coefficient \eqref{B0zeta}.
For the definition of the upscaled relative errors in the last column, see \eqref{nO0}.
All entries are rounded to five decimals.
}}
\label{taba0}
\vspace{1cm}
\begin{tabular}{|| r || r | r | c | c ||} \hline\hline
 &     &    &    &      \\[-3mm]
\multicolumn{1}{||c||}{$n$} & \multicolumn{1}{c |}{$a_n$}    & \multicolumn{1}{c |}{$\sim\,\tilde{a}_n$} & $n\cdot\cO(1/n)$  & $a_{0n}\,/\,a_n$ \\[1mm] \hline\hline 
 &     &    &     &     \\[-3mm]
$ 4 $&$13 $&$25.260 $&$1.9414 $&$0.30769 $\\
$ 13 $&$159\,570 $&$192740. $&$2.2369 $&$0.13641 $\\
$ 21 $&$1.1360\times 10^9 $&$1.2771\times 10^9 $&$2.3200 $&$0.12142 $\\
$ 30 $&$3.0785\times 10^{13} $&$3.3436\times 10^{13} $&$2.3791 $&$0.11559 $\\
$ 1\,000 $&$2.7827\times 10^{512} $&$2.7897\times 10^{512} $&$2.5342 $&$\
0.10361 $\\
$ 2\,000 $&$5.0158\times 10^{1028} $&$5.0222\times 10^{1028} $&$2.5368 $&$\
0.10343 $\\
$ 3\,000 $&$1.3908\times 10^{1545} $&$1.3919\times 10^{1545} $&$2.5377 $&$\
0.10337 $\\
$ 4\,000 $&$4.6005\times 10^{2061} $&$4.6035\times 10^{2061} $&$2.5381 $&$\
0.10334 $\\
$ 5\,000 $&$1.6764\times 10^{2578} $&$1.6772\times 10^{2578} $&$2.5384 $&$\
0.10333 $\\
$ 6\,000 $&$6.4939\times 10^{3094} $&$6.4966\times 10^{3094} $&$2.5386 $&$\
0.10332 $\\
$ 7\,000 $&$2.6241\times 10^{3611} $&$2.6251\times 10^{3611} $&$2.5387 $&$\
0.10331 $\\
$ 8\,000 $&$1.0937\times 10^{4128} $&$1.0940\times 10^{4128} $&$2.5388 $&$\
0.10330 $\\
$ 9\,000 $&$4.6670\times 10^{4644} $&$4.6683\times 10^{4644} $&$2.5388 $&$\
0.10330 $\\
$ 10\,000 $&$2.0290\times 10^{5161} $&$2.0295\times 10^{5161} $&$2.5389 $&$\
0.10329 $\\ [-3mm]
 &     &    &     &    \\           
\hline\hline
\end{tabular} 
\caption[Asymptotics counts $a_n$ of signature forests]{\small\emph{
Total counts $a_n$ of signature forests $\fS$ for the same $n$ as in table \ref{taba0}.
Column three lists the asymptotics $a_n\!\sim\!\tilde{a}_n$ of \eqref{gammansim}, with leading coefficient \eqref{Bzeta}.
For the definition of the upscaled relative errors in column four, see \eqref{nO}.
The last column documents a limiting prevalence \eqref{a0n/an} of approximately $10\%$, for tree among forest counts $a_{0n}/a_n$\,, alias irreducibility among Hamiltonian Sturm permutations, in the sense of \eqref{stack}.
}}
\label{taba}
\end{table}

\begin{table}[p]
\begin{center}
\begin{tabular}{|| c || r | r | r | r ||} \hline\hline
 &     &    &     &     \\[-3mm]
$r$ & $n\!=\!4$    & $n\!=\!13$ & \multicolumn{1}{c |}{$n\!=\!21$} & \multicolumn{1}{c ||}{$n\!=\!30$} \\[1mm] \hline\hline 
 &     &    &     &     \\[-3mm]
 1 & 3    &  924   &  184\,756   & 77\,558\,760    \\
2 & 1    & 1\,588    &  527\,900   & 313\,889\,477    \\
3 &     &   3\,704  &  2\,171\,110   & 1\,829\,806\,591    \\
4 &     &   5\,212  & 6\,281\,200    & 8\,195\,875\,855    \\
5 &     &   5\,105  & 13\,549\,060    & 28\,454\,717\,383    \\  
6 &     &   3\,236  & 21\,907\,952    & 78\,390\,368\,975    \\  
7 &     &   1\,472  & 27\,261\,708    & 173\,646\,331\,179    \\  
8 &     &   424  &  26\,226\,568   & 313\,605\,452\,308    \\  
9 &     &   91  &  19\,806\,459   & 466\,123\,210\,176    \\  
10 &     & 10    & 11\,730\,094    & 575\,164\,508\,971    \\ 
11 &     & 1    &  5\,503\,061   & 592\,771\,070\,809    \\
12 &     &     &  2\,026\,036   & 513\,031\,411\,787    \\
13 &     &     &  590\,777   & 374\,178\,954\,189    \\
14 &     &     &  132\,888   & 230\,691\,526\,573    \\
15 &     &     &  23\,502   & 120\,382\,876\,553    \\  
16 &     &     &  3\,024   & 53\,231\,864\,376    \\  
17 &     &     &  307   & 19\,926\,472\,616    \\  
18 &     &     &  18   & 6\,310\,154\,319    \\  
19 &     &     &  1   & 1\,683\,764\,021    \\  
20 &     &     &     &  377\,652\,125   \\   
21 &     &     &     &  70\,549\,993   \\
22 &     &     &     &  10\,941\,749   \\
23 &     &     &     &  1\,378\,689   \\
24 &     &     &     &  141\,552   \\
25 &     &     &     &  11\,100   \\  
26 &     &     &     &  703   \\  
27 &     &     &     &  27  \\  
28 &     &     &     &  1   \\ [-3mm]
 &     &    &     &     \\           
\hline\hline
 &     &    &     &     \\[-3mm]
$a_{0n}$ & $4$    & $21\,767$ & $137\,926\,421$ & $3\,558\,470\,490\,857$ \\[1mm] \hline\hline 
\end{tabular} 
\end{center} 
\caption[Refined tree counts $f_{0,nr}$ for $n=4,13,21,30$]{\small\emph{Coefficients $f_{0,nr}$ which count full lap signature $\frak{S}$ of \emph{tree type} with $r$ ODE-centers $\odot$ and $q\!=\!n\!-\!r$ pairs of lap numbers $\ell$, alias ODE orbits of minimal period $2/\ell$.
Columns sample choices of $n\in\{4,13,21,30\}$.
Rows indicate $r=1,\ldots,n-2$.
The entries of the first column $n\!=\!4$ refer to tree entries $\checkmark$ in table \ref{tablap} with $r\!+\!q\!=\!4$; specifically these are the $f_{0,41}\!=\!3$ entries 1.3 and the single $f_{0,42}\!=\!1$ entry 2.2(iv).
The entries $f_{0,nr}\!=\!1$ for $r\!=\!n\!-\!2\!\geq\!2$ arise because such a full lap signature tree requires an annular lap list $\underline{S}\!=\!\{1,1\}$ of $q\!=\!2$ periodic orbits which encircle all $r$ centers.
The last row provides the column sums $a_{0n}=\sum_r\,f_{0,nr}$ already listed in table \ref{taba0}.
}}
\label{tabf0}
\end{table}

\begin{table}[p]
\begin{center}
\begin{tabular}{|| c || r | r | r | r ||} \hline\hline
 &     &    &     &     \\[-3mm]
$r$ & $n\!=\!4$    & $n\!=\!13$ & $n\!=\!21$ & $n\!=\!30$ \\[1mm] \hline\hline 
 &     &    &     &     \\[-3mm]
1 & 3  & 924  & 184\,756  & 77\,558\,760  \\
2 & 6  & 4\,760  & 1\,391\,720  & 773\,201\,629  \\
3 & 3  & 14\,758  & 6\,921\,760  & 5\,235\,566\,429  \\
4 & 1  & 28\,576  & 24\,187\,160  & 26\,786\,857\,718  \\
5 &   & 38\,204  & 62\,701\,430  & 105\,897\,127\,937  \\  
6 &   & 34\,772  & 122\,609\,392  & 331\,668\,939\,331  \\  
7 &   & 22\,792  & 185\,315\,596  & 836\,256\,161\,707  \\  
8 &   & 10\,384  & 218\,755\,440  & 1\,722\,306\,886\,392  \\  
9 &   & 3\,480  & 204\,472\,860  & 2\,928\,442\,547\,844  \\  
10 &   & 780  & 152\,048\,900  & 4\,148\,616\,110\,119  \\ 
11 &   & 132  & 90\,681\,558  & 4\,930\,805\,575\,355  \\
12 &   & 12  & 43\,343\,052  & 4\,945\,836\,031\,732  \\
13 &   & 1  & 16\,682\,601  & 4\,204\,912\,182\,497  \\
14 &   &   & 5\,128\,860  & 3\,041\,157\,254\,879  \\
15 &   &   & 1\,264\,640  & 1\,875\,443\,961\,273  \\  
16 &   &   & 244\,288  & 987\,952\,987\,982  \\  
17 &   &   & 37\,434  & 444\,800\,338\,056  \\  
18 &   &   & 4\,236  & 171\,181\,762\,911  \\  
19 &   &   & 380  & 56\,231\,234\,407  \\  
20 &   &   & 20  & 15\,741\,848\,294  \\ 
21 &   &   & 1  & 3\,739\,996\,939  \\
22 &   &   &   & 751\,588\,365  \\
23 &   &   &   & 126\,607\,065  \\
24 &   &   &   & 17\,798\,570  \\
25 &   &   &   & 2\,045\,080  \\  
26 &   &   &   & 192\,231  \\  
27 &   &   &   & 13\,887  \\  
28 &   &   &   & 812  \\  
29 &   &   &   & 29  \\  
30 &   &   &   & 1  \\ [-3mm]
 &     &    &     &     \\                
\hline\hline
 &     &    &     &     \\[-3mm]
$a_n$ & $13$    & $159\,575$ & $1\,135\,976\,084$ & $30\,784\,762\,378\,231$ \\[1mm] \hline\hline 
\end{tabular} 
\end{center} 
\caption[Refined forest counts $f_{nr}$ for $n=4,13,21,30$]{\small\emph{Coefficients $f_{nr}$ which count full lap signatures $\frak{S}$ of \emph{forest type} with $r$ centers and $q\!=\!n\!-\!r$ orbits of (not necessarily minimal) period 2.
Columns sample the same choices $n\in\{4,13,21,30\}$ as table \ref{tabf0}.
Rows indicate $r=1,\ldots,n$.
The first column $n\!=\!4$ refers to \emph{all} entries in table \ref{tablap} with $r\!+\!q\!=\!4$; specifically these are the entries $f_{4r}=3,6,3,2$ for $r=1,2,3,4$, respectively.
The entries $f_{0,nr}\!=\!1$ for $r\!=\!n\!\geq\!2$ refer to the trivial case $q\!=\!0$ without any 2-periodic orbits.
The last row provides the column sums $a_n=\sum_r\,f_{nr}$ already listed in table \ref{taba}.
}}
\label{tabf}
\end{table}

The generating functions $a(z)=\sum_n\,a_nz^n$ count the full lap signatures $\fS$ with $r\!\geq\!1$ ODE-centers $\odot$ and $q\!\geq\!0$ pairs of lap number entries, such that $n\!=\!r\!+\!q$; see \eqref{az}.
Definition \ref{deftree} has represented each signature as a signature forest.
An additional index 0 in $a_0(z)$ and $a_{0n}$ indicates the analogous counts of the connected forests, i.e. of signature trees.

Table \ref{taba0} lists some select counts $a_{0n}$ of signature trees, for $n$ from 4 to 10\,000 (first column).
The second column lists actual counts $a_{0n}$.
Column three compares those to the asymptotics \eqref{gammansim}, i.e.\  $a_{0n}\!\sim\!\tilde{a}_{0n}$, of theorem \ref{thmasa} with leading coefficient \eqref{B0zeta}.
The final column upscales the relative errors of order $\cO(1/n)$ by a factor $n$ to obtain asymptotically constant corrections
\begin{equation}
\label{nO0}
n\cdot \cO_0(1/n)\ :=\ n\cdot\left(\frac{a_{0n}}{\tilde{a}_{0n}}-1\right)\,.
\end{equation}

Table \ref{taba}, in contrast, presents select actual counts $a_n$ for signature forests, tress included.
Column three compares $a_n$ to the asymptotics \eqref{gammansim}, i.e.\  $a_n\!\sim\!\tilde{a}_n$, of theorem \ref{thmasa} with leading coefficient \eqref{Bzeta}.
The final column lists the upscaled relative errors again, i.e.
\begin{equation}
\label{nO}
n\cdot \cO_0(1/n)\ :=\ n\cdot\left(1-\frac{a_n}{\tilde{a}_n}\right)\,.
\end{equation}
Quotients $a_{0n}/a_n$ of the corresponding counts in the second columns of tables \ref{taba0} and \ref{taba}, as listed in the last column of table \ref{taba}, of course confirm the $\cO(1/n)$ asymptotics
\begin{equation}
\label{a0n/an}
\lim_{n\rightarrow\infty} \,\frac{a_{0n}}{a_n} \,=\, \frac{B_0(\zeta)}{B(\zeta)}\,=\, (1-a_0(\zeta))^2\,=\,0.103\,256\,667\,322\,663\ldots\ .
\end{equation}
In \cite{wol17}, counts $a_n^*$ of \emph{all Sturm permutations} with $N\!=\!2n\!+\!1$ equilibria have been derived for $1\leq n\leq20$.
Comparison with our counts $a_n$ seems to suggest crude asymptotic estimates like $a_n^*\sim a_n^2$\,.
The true asymptotic growth of $a_n^*$ remains unknown.

We now turn to the refined counts $b_{0,nr}$ and $b_{nr}$ of full lap signatures $\fS$ with $1\leq r\leq n$ ODE-centers $\odot$ and $q\!=\!n\!-\!r$ pairs of lap number entries.
See \eqref{ab}--\eqref{f0}.
For columns $n=4,13,21,30$, tables \ref{tabf0} and \ref{tabf} provide raw data of $f_{0,nr}\!=\!b_{0,r,n-r}$ and $f_{nr}\!=\!b_{r,n-r}$\,, respectively, in rows $r$.
As usual, the additional index 0 in table \ref{tabf0} refers to lap signatures $\fS$ of tree type, alias irreducible Hamiltonian Sturm permutations $\sigma$ in the sense of \eqref{stack}.
Omitted indices 0 in table \ref{tabf} refer to forests $\fS$ and include stacking; see definition \ref{deftree}.

The columns $n\!=\!4$ in tables \ref{tabf0} and \ref{tabf} summarize the counts resulting from the explicit signature trees and forests $\fS$ in table \ref{tablap}.
We obtain 3 cases and 1 case, i.e.\ the two entries in tree table \ref{tabf0} for $r\!=\!1$ and $r\!=\!2$, respectively, from the checkmarked cases 1.3(i)-(iii) and 2.2(iv) of table \ref{tablap}, which are all symmetric.
The $3,6,3,1$ cases, i.e.\ the four entries for $1\!\leq\!r\!\leq\!n\!=\!4$ in forest table \ref{tabf} arise from cases 1.3(i)-(iii), 2.2(i)-(iv), 3.1(i)-(ii), 4.0(i) in table \ref{tablap}, respectively.
Here we have to re-insert omitted trivial reflections in the nonsymmetric cases 2.2(i), 2.2(ii), and 3.1(i), to reach the full counts.
Facing rapidly growing totals of $a_5\!=\!35,\ a_6\!=\!96$, and so on, we refrained from elaborating on such explicit listings of lap signatures $\fS$ any further.

The single counts $f_{0,n,n-2}\!=\!f_{nn}\!=\!1$, at maximal $r$ with nonzero entries, arise as follows.
In the forest table \ref{tabf}, the maximal number of ODE-centers $\odot$ is $r\!=\!n$.
In that case, 
$q\!=\!n\!-\!r\!=\!0$ 
prohibits any pairs of lap numbers, alias 2-periodic orbits.
Definition \ref{defSign} therefore leaves us with the unique signature $\fS$ of alternating ODE-saddles $\bullet$ and ODE-centers $\odot$\,.

\begin{table}[t!]
\centering
\begin{tabular}{|| r || r | r | r || r | r | r ||} \hline\hline
 &   &   &   &   &   &       \\[-3mm]
\multicolumn{1}{||c||}{$n$} & \multicolumn{1}{c |}{$\boldsymbol{\mu}_{0n}$}    & \multicolumn{1}{c |}{$\tilde{\boldsymbol{\mu}}_{0n}$} & \multicolumn{1}{c ||}{$\cO_{0\mu}$} & 
\multicolumn{1}{c |}{$\boldsymbol{\sigma}_{0n}^2$}    & \multicolumn{1}{c |}{$\tilde{\boldsymbol{\sigma}}_{0n}^2$} & \multicolumn{1}{c ||}{$\cO_{0\sigma^2}$}  \\[1mm] \hline\hline 
 &   &   &   &   &   &       \\[-3mm]
 4 & 1.2500 & 1.2071 & -0.1716 & 0.1875 & 0.7075 & 2.0800 \\
 7 & 2.0769 & 2.3111 & 1.6390 & 1.1095 & 1.2679 & 1.1090 \\
 10 & 3.2198 & 3.4150 & 1.9524 & 1.9527 & 1.8283 & -1.2444 \\
 13 & 4.3932 & 4.5190 & 1.6351 & 2.5622 & 2.3887 & -2.2557 \\
 16 & 5.5443 & 5.6229 & 1.2576 & 3.0659 & 2.9491 & -1.8683 \\
 19 & 6.6723 & 6.7269 & 1.0382 & 3.5797 & 3.5095 & -1.3346 \\
 21 & 7.4175 & 7.4629 & 0.9532 & 3.9330 & 3.8831 & -1.0485 \\
 24 & 8.5302 & 8.5668 & 0.8803 & 4.4757 & 4.4435 & -0.7740 \\
 27 & 9.6396 & 9.6708 & 0.8415 & 5.0272 & 5.0039 & -0.6311 \\
 30 & 10.7475 & 10.7748 & 0.8178 & 5.5827 & 5.5643 & -0.5541 \\[-3mm]
 &   &   &   &   &   &         \\           
\hline\hline
\end{tabular} 
\caption[Asymptotic means and variances of signature trees]{\small\emph{
Means $\boldsymbol{\mu}_{0n}$ and variances $\boldsymbol{\sigma}_{0n}^2$ of signature tree counts $p_{0,nr}\!=\!f_{0,nr}/a_{0n}$\,, for select rows of $4\!\leq\! n\!\leq\!30$.
Column three lists the asymptotics \eqref{tmu0n} of theorem \ref{thmbasy}.
For the definition of the upscaled additive error $\cO_{0\mu}$ in column four, see \eqref{nO0mu}.
Columns five to seven list the analogous data for variances $\boldsymbol{\sigma}_{0n}^2$ and their asymptotics \eqref{tsig0n}, \eqref{nO0sig}.
All entries are rounded to five decimals.
}}
\label{tabGauss0}
%
\vspace{1cm}
\begin{tabular}{|| r || r | r | r || r | r | r ||} \hline\hline
 &   &   &   &   &   &       \\[-3mm]
\multicolumn{1}{||c||}{$n$} & \multicolumn{1}{c |}{$\boldsymbol{\mu}_n$}    & \multicolumn{1}{c |}{$\tilde{\boldsymbol{\mu}}_n$} & \multicolumn{1}{c ||}{$\cO_\mu$} & 
\multicolumn{1}{c |}{$\boldsymbol{\sigma}_n^2$}    & \multicolumn{1}{c |}{$\tilde{\boldsymbol{\sigma}}_n^2$} & \multicolumn{1}{c ||}{$\cO_{\sigma^2}$}  \\[1mm] \hline\hline 
 &   &   &   &   &   &       \\[-3mm]
 4 & 2.1538 & 2.0389 & -0.4600 & 0.7456 & 1.0293 & 1.1351 \\
 7 & 3.1791 & 3.1428 & -0.2540 & 1.4306 & 1.5897 & 1.1138 \\
 10 & 4.2484 & 4.2468 & -0.0167 & 2.0666 & 2.1501 & 0.8356 \\
 13 & 5.3393 & 5.3507 & 0.1486 & 2.6749 & 2.7105 & 0.4635 \\
 16 & 6.4403 & 6.4547 & 0.2304 & 3.2571 & 3.2709 & 0.2207 \\
 19 & 7.5446 & 7.5587 & 0.2662 & 3.8267 & 3.8313 & 0.0866 \\
 21 & 8.2815 & 8.2946 & 0.2767 & 4.2029 & 4.2049 & 0.0424 \\
 24 & 9.3868 & 9.3986 & 0.2833 & 4.7648 & 4.7653 & 0.0127 \\
 27 & 10.4920 & 10.5025 & 0.2852 & 5.3255 & 5.3257 & 0.0046 \\
 30 & 11.5970 & 11.6065 & 0.2854 & 5.8859 & 5.8861 & 0.0051 \\[-3mm]
 &   &   &   &   &   &         \\           
\hline\hline
\end{tabular} 
\caption[Asymptotic means and variances of signature forests]{\small\emph{
Means $\boldsymbol{\mu}_n$ and variances $\boldsymbol{\sigma}_n^2$ of forest counts in the same layout as table \ref{tabGauss0}.}}
\label{tabGauss}
\end{table}

The tree requirement on lap signatures $\fS$ in table \ref{tabf0}, eliminates that alternating saddle-center option altogether, for $n\geq2$.
Indeed, the signature forest then decomposes into $r\!=\!n\!\geq\!2$ disconnected singletons.
The tree requirement in fact mandates an outermost parenthesis pair $\fS=\bullet(\underline{S}_1\ldots\overline{S}_1)\bullet$.
See definitions \ref{defSign}, \ref{deftree}.
Since $r>1$, that pair cannot be central.
As an annular pair, the lap list $\underline{S}_1$ must be nonempty and of even length.
In other words, $r\!=\!n\!-\!q\!\leq\! n\!-\!2$.
By definition \ref{defSign}(vii), \eqref{laplistbc}, the maximal case $r\!=\!n\!-\!2$ of nonzero $f_{0,nr}$ is therefore uniquely realized by a single nonempty lap list $\underline{S}_1\!=\!\{1,1\}$ and an interior chain of $r$ ODE-centers alternating with $r\!-\!1$ ODE-saddles:
\begin{equation}
\label{r=n-2}
\fS=\bullet(\{1,1\}(\odot)\bullet(\ldots)\bullet(\odot) \{1,1\})\bullet \,.
\end{equation}

The last rows of tables \ref{tabf0}, \ref{tabf} list the normalizing column sums $a_{0n}$ and $a_n$\,.
Note how the entry patterns in the right three columns discretize right open parabolas.
Since the length in decimal digits discretizes the logarithm, this phenomenon can already be seen as a first indication of the Gaussian asymptotics for the probability distributions $p_{0,nr}=f_{0,nr}/a_{0n}$ and $p_{nr}=f_{nr}/a_n$\,.
See theorems \ref{thmGauss} and \ref{thmbasy}.

Tables \ref{tabGauss0} and \ref{tabGauss} illustrate the asymptotics of means $\boldsymbol{\mu}$ and variances $\boldsymbol{\sigma}^2$ for the refined counts of tree and forest signatures $\fS$, respectively, and for select ten rows of low $4\leq n\leq30$.
Again, all entries are rounded to five decimals.

Consider the tree case of table \ref{tabGauss0}, first.
The second column lists actual means $\boldsymbol{\mu}_{0n}$ of the probability distributions $\bP(X_{0n}=r)=p_{0,nr}:=f_{0,nr}/a_{0n}$\,; see \eqref{pnu}--\eqref{Y0n}.
Column three compares those to the $\cO_{0\mu}(1/n)$ convergence of theorem \ref{thmbasy} towards
\begin{equation}
\label{tmu0n}
\tilde{\boldsymbol{\mu}}_{0n}:=\mathfrak{m}_1\,n+\mathfrak{m}_{0,0}\,.
\end{equation}
Column four upscales the additive errors of order $\cO_{0\mu}(1/n)$ by a factor $n$ to obtain signed correcting constants
\begin{equation}
\label{nO0mu}
\cO_{0\mu}:=n\cdot(\tilde{\boldsymbol{\mu}}_{0n}-\boldsymbol{\mu}_{0n})\,.
\end{equation}

\begin{figure}[p]
\begin{center}
\includegraphics[width = 0.69\textwidth]{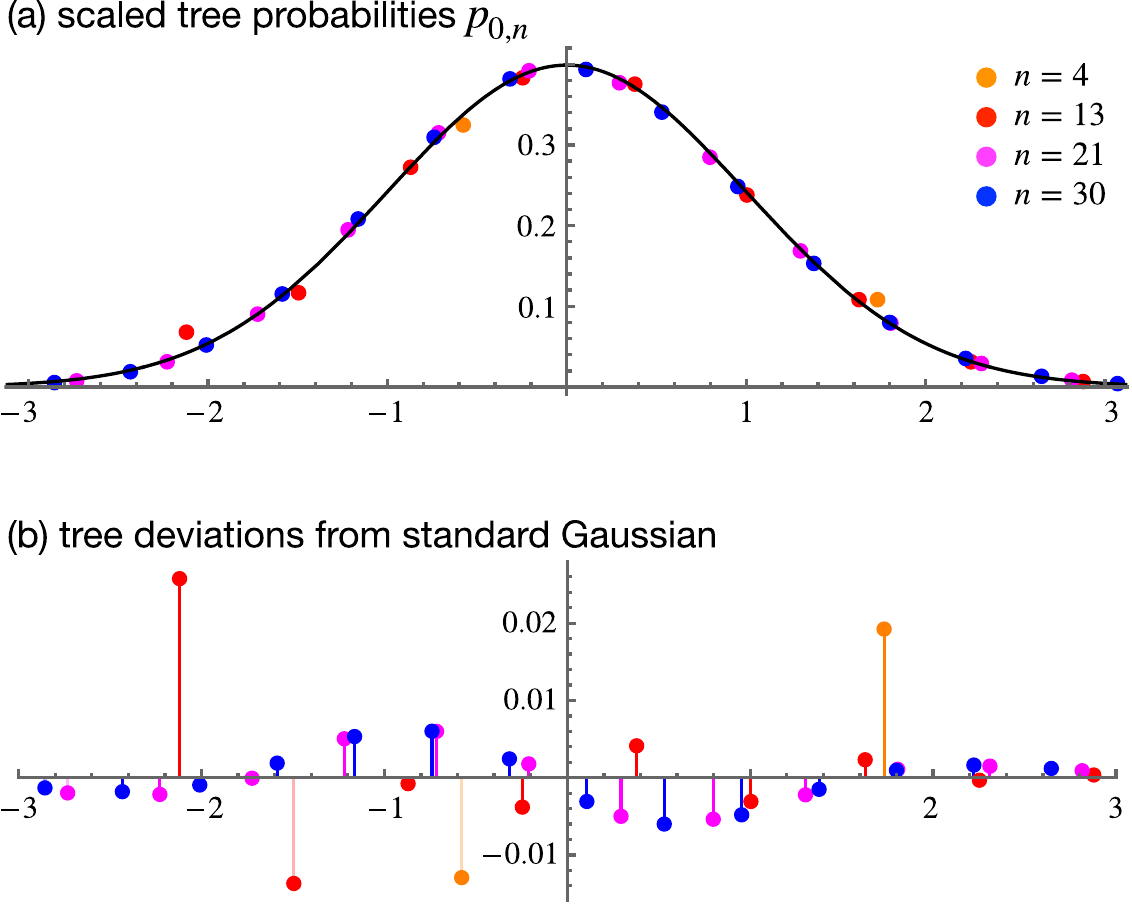}
\end{center}
\caption[Gaussian tree frequencies]{\small\emph{Gaussian frequency distribution $p_{0,n}$ of tree signatures $\fS$ with respect to $1\leq r\leq n-2$ ODE-centers $\odot$.
The distribution is scaled to mean 0 and standard deviation 1, as in \eqref{Y0n}, \eqref{p0scale}.
Colorcoding: $n\!=\!4$ (orange), $n\!=\!13$ (red), $n\!=\!21$ (magenta), and $n\!=\!30$ (blue).
(a) comparison with standard Gaussian. (b) enlarged deviations oscillate around standard Gaussian, and decrease with increasing $n$.}}
\label{figgauss0}
\vspace{1cm}
\begin{center}
\includegraphics[width = 0.69\textwidth]{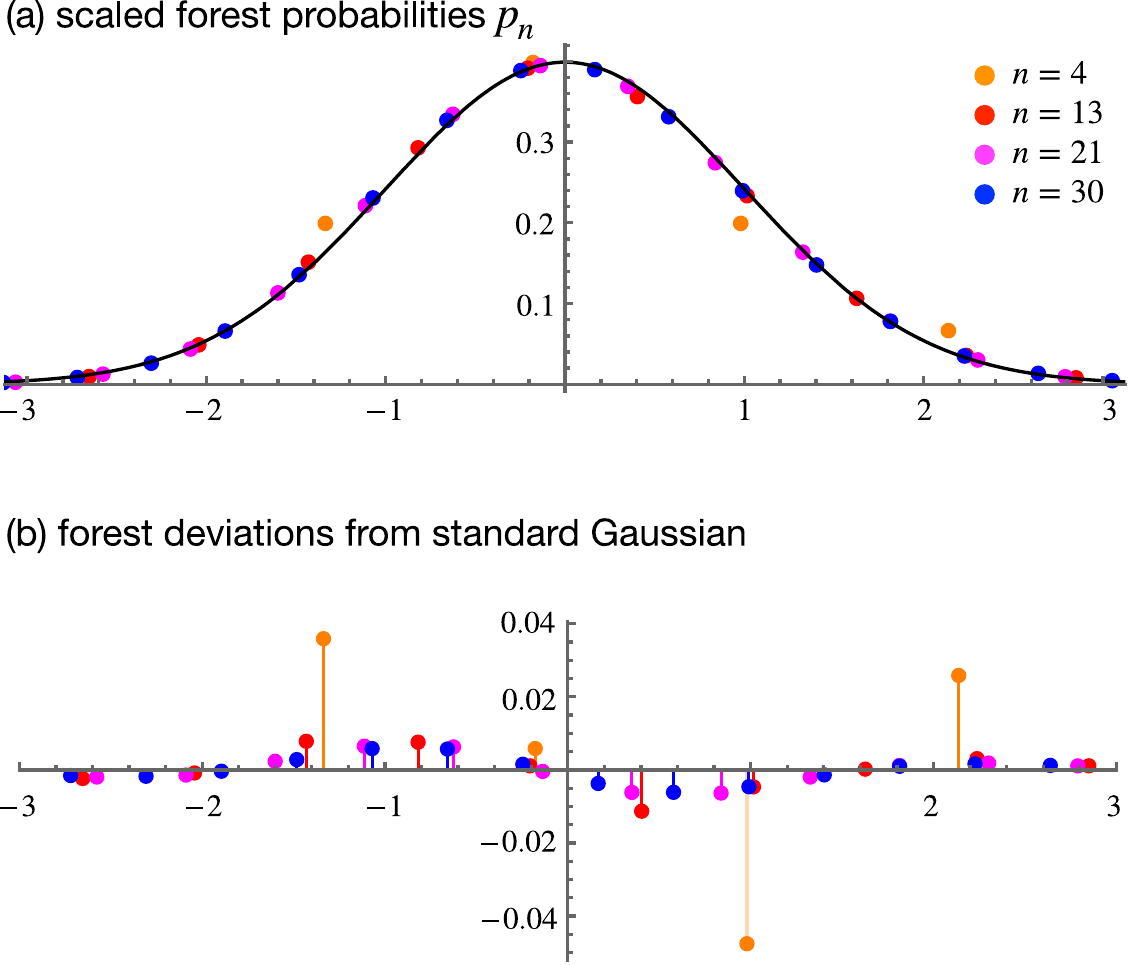}
\end{center}
\caption[Gaussian forest frequencies]{\small\emph{Gaussian frequency distribution $p_n$ of forest signatures $\fS$, for $1\leq r\leq n$, analogously to figure \ref{figgauss0}.
Scaling as in \eqref{Yn}, \eqref{p0scale}.}}
\label{figgauss}
\end{figure}

Columns five to seven list the analogously defined data for variances $\boldsymbol{\sigma}_n^2$; specifically
\begin{align}
\label{tsig0n}
    \tilde{\boldsymbol{\sigma}}_{0n}^2&:=\mathfrak{v}_1\,n+\mathfrak{v}_{0,0}\,;
   \\
\label{nO0sig}
    \cO_{0\sigma}&:=n\cdot(\tilde{\boldsymbol{\sigma}}_{0n}^2-\boldsymbol{\sigma}_{0n}^2)\,.  
\end{align}
Omitting an index 0, table \ref{tabGauss} illustrates the forest case of theorem \ref{thmbasy}, by the same layout.

Here are some conclusions supported by the two tables.
Expectedly, means $\boldsymbol{\mu}$ and variances $\boldsymbol{\sigma}$ increase slightly, when passing from the more restrictive tree case to forest signatures.
The additional flexibility thus gained, by concatenation of tree signatures, favors the placement of additional ODE-centers $\odot$ in $\fS$.
Comparing columns two and three, for means $\boldsymbol{\mu}$, and columns five and six, for variances $\boldsymbol{\sigma}^2$, confirms the convergences of order $\cO(1/n)$ which theorem \ref{thmbasy} has established for full lap signatures of, both, tree and forest type.
Upscaled errors in columns four seem to indicate asymptotic overestimation of means by approximately $0.8/n$ and $0.3/n$ in the tree and forest case, respectively.
Variances seem to be underestimated by approximately $0.5/n$, in the tree case, but overestimated by as little as $0.005/n$, in the forest case.

We conclude with illustrations of the Gaussian asymptotics established in theorem \ref{thmGauss}.
For tree signatures $\fS$, alias irreducible Hamiltonian Sturm permutations $\sigma$, and select fixed $n=4,13,21,30$, we scale the normalized frequency counts $p_{0,nr}=f_{0,nr}/a_{0n},\ 1\leq n\leq n-2$.
Our scaling is analogous to the normalization of $Y_{0n}$ in \eqref{Y0n}.
With means $\mu_{0n}$ and standard deviations $\sigma_{0n}$ from table \ref{tabGauss0}, scaling maps the point $(r\,,p_{0,nr})$ to 
\begin{equation}
\label{p0scale}
(r,\,p_{0,nr})\ \mapsto\ \left(\frac{r-\mu_{0n}}{\sigma_{0n}},\ \sigma_{0n}\cdot p_{0,nr}\right)\,.
\end{equation}
The factor $\sigma_{0n}$ preserves interpretation as a discretized version of a probability density function.
We colorcode $n\!=\!4$ (orange), $n\!=\!13$ (red), $n\!=\!21$ (magenta), and $n\!=\!30$ (blue).
Theorem \ref{thmGauss} only asserts $\cO(1/\sqrt{n})$ sup-norm convergence of the associated cumulative probability distribution functions to the cumulative, i.e.\ integrated, standard Gaussian.
For the interpolating probability density step functions, comparison with the standard Gaussian still leads to reasonably small errors; see figure \ref{figgauss0}(a).
Enlarging deviations from the standard Gaussian, in (b), illustrates decreasing errors of oscillatory type.

Figure \ref{figgauss} provides the analogous illustrations of theorem \ref{thmGauss}, for forest signatures $\fS$. 
The scaling \eqref{Yn} to mean 0 and standard deviation 1 is performed by \eqref{p0scale} again, with all indices 0 omitted.
Although the sample size of the forest case exceeds the tree case by a factor $a_n/a_{0n}\sim(1-a_0)^{-2}$ of almost 10, asymptotically, errors do not increase significantly.

\section{Periodic boundary conditions and reversibility}\label{Per}

We conclude with a few remarks on PDE \eqref{gxuux} under 2-periodic boundary conditions $\cP$\,:
\begin{equation}
\label{perbc}
\mathbf{u}\large{\vert}_{\mathbf{x}=0}^{\mathbf{x}=2}\, =\,\mathbf{u_x}\large{\vert}_{\mathbf{x}=0}^{\mathbf{x}=2}\, =\, 0\,, \quad\textrm{alias}\quad \mathbf{x}\in\bS^1=\bR /2\bZ\,;
\end{equation}
compare \eqref{pbc}.
See \eqref{SPchain} for notations of  various classes of nonlinearities $\mathbf{g}$ under periodic boundary conditions $\cP$, in analogy with the Neumann cases of superscripts $\cN$ addressed so far.
The concept of global attractors $\cA\!=\!\cA_\mathbf{g}^\cP$ in suitable Sobolev spaces $\mathbf{X}\hookrightarrow C^1$ readily extends to periodic boundary conditions, for example under the same sufficient dissipativity condition \eqref{diss} of \cite{mana97}.
For further results on periodic boundary conditions see also \cite{anfi88, fm-p, mat88} and, for some results of collaboration by the present authors on this subject, \cite{firowo04,firowo12b, roc25,rofiln}.

Periodic boundary conditions allow for time-periodic solutions $\mathbf{u}(\mathbf{t},\cdot)$, as a main novel feature.
Explicit Lyapunov functionals like \cite{zele, mat88,firaro,labe,lafi} therefore fail to carry over from the Neumann case $\cN$ to $\cP$. 
For example, \emph{any autonomous planar ODE flow} can be embedded into the PDE dynamics of \eqref{gxuux}, for suitable choices of the nonlinearity $\mathbf{g}\!\in\!\SxP$; see \cite{safi}.
In particular, homoclinic and \emph{saddle-saddle heteroclinic orbits}, between hyperbolic equilibria of the same Morse index, may also occur.
Although invariant planes are not a typical feature of $\mathbf{g}\!\in\! \SxP$, and global attractors $\cA_\mathbf{g}^\cP$ may be of any finite dimension, a Poincaré-Bendixson theorem rules asymptotic $\boldsymbol{\alpha}$- and $\boldsymbol{\omega}$-limit sets of uniformly bounded solutions $\mathbf{u}(\mathbf{t},\cdot)\in\cA^\cP$; see \cite{fm-p}.
The underlying reason is the straightforward extension of zero number dropping \eqref{zdrop}, from $\cN$ to $\cP$.

Transversality \eqref{transv} of stable and unstable manifolds among periodic orbits and equilibria therefore persists \cite{czro08,jora10a}, with the notable exception of the saddle-saddle connections mentioned above.
That exception poses a fundamental obstacle to a purely combinatorial description of connection graphs, here among equilibria and periodic orbits.
Obviously, for lack of distinguished Neumann boundaries in the periodic case $\cP$, we cannot rely on Sturm permutations based on boundary orders $\sigma$ like \eqref{sigma0}, \eqref{sigma1}. 
Well, suppose we replace those by relative winding numbers of temporal snapshots
\begin{equation}
\label{wdg}
\mathbf{x}\ \mapsto\ \left(\mathbf{x},\mathbf{u}(\mathbf{t},\mathbf{x}),\mathbf{u}_\mathbf{x}(\mathbf{t},\mathbf{x})\right)\ \in\ \bS^1\times\bR^2\,.
\end{equation}
Still, the occurrence of saddle-saddle connections among equilibria, at infinite periods, might disrupt homotopy invariance of such descriptions.

For such reasons, we restricted attention to nonlinearities $\mathbf{g}=\mathbf{g}(\mathbf{u},\mathbf{u}_\mathbf{x})$ in the classes 
\begin{equation}
\label{gP}
\SH\subset\SO\subset\SSOP\,, 
\end{equation}
which do not depend on $\mathbf{x}$, explicitly.
Equivalently, we may observe that the semiflow of PDE \eqref{gxuux} on $\mathbf{X}$ then becomes \emph{equivariant} under the action of the group $\bS\bO\!\cong\!\bS^1$, which acts by spatial shifts $\boldsymbol{\xi}\!\in\!\bS^1$ of $\mathbf{x}$.
In other words, $\mathbf{u}(\mathbf{t},\mathbf{x})$ is a solution whenever $\mathbf{u}(\mathbf{t},\mathbf{x}+\boldsymbol{\xi})$ is, for any fixed $\boldsymbol{\xi}$\,.
See \cite{golste88,golste23} for a broader background on equivariance.

Spatial shifts $\boldsymbol{\xi}\!\in\!\bS^1$ ensure that each spatially non-homogeneous equilibrium $\mathbf{v}$ generates a circle's worth of PDE equilibria.
We call such circles $\bS\bO\mathbf{v}$ \emph{frozen waves} $\cF$, to emphasize that they occur as a single symmetry related group orbit. 

By $\bS\bO$-equivariance and zero-number dropping \eqref{zdrop}, nonstationary time-periodic solutions  $\mathbf{u}(\mathbf{t},\mathbf{x})$ turn out to be \emph{rotating waves} $\mathbf{v}\!\in\!\cR_\mathbf{g}^\cP$ at \emph{wave speed} $\mathbf{c}\neq0$\,, i.e.
\begin{equation}
\label{rot}
 \mathbf{u}(\mathbf{t},\mathbf{x})=\mathbf{v}(\mathbf{x}-\mathbf{c}\mathbf{t})\,.
\end{equation}
In other words, $\mathbf{v}(\mathbf{x})$ solves the ODE
\begin{equation}
\label{gvpc}
\begin{aligned}
    \mathbf{v}_\mathbf{x}&=\mathbf{p} \,,  \\
    \mathbf{p}_\mathbf{x}&=-\mathbf{g}(\mathbf{v},\mathbf{p})+\mathbf{c}\,\mathbf{p}  
\end{aligned}
\end{equation}
with wave speed parameter $\mathbf{c}\neq0$.
Periodic boundary conditions \eqref{perbc} require $\mathbf{v}$, and hence $\mathbf{p}$, to be 2-periodic of minimal period
\begin{equation}
\label{TlP}
0<T_\mathbf{g}(\mathbf{v})=2/\ell\,,
\end{equation}
for some positive integer $\ell\!\in\!\bN$.
Spatially non-homogeneous PDE equilibria $\mathbf{v}$ can be subsumed as ``rotating waves'' at wave speed $\mathbf{c}\!=\!0$, which we call \emph{frozen waves} $\mathbf{v}\!\in\!\cF_\mathbf{g}^\cP$; see \cite{fiedlerhabil}.
Note how rotating and frozen waves, alike, occur as nontrivial circular group orbits $\bS\bO\mathbf{v}$.
Sometimes, they are therefore also called \emph{relative equilibria}.
Spatially homogeneous PDE equilibria, i.e.\ equilibria $(\mathbf{v},0)$ of ODE \eqref{gvpc}, are denoted as $\mathbf{v}\!\in\!\cE_\mathbf{g}^\cP$\,.

Such notation is quite analogous to decomposition \eqref{AEH} in the Neumann case $\cN$. 
Omitting indices $\mathbf{g}$, the global attractor $\cA^\cP$ for periodic boundary conditions now decomposes as
\begin{equation}
\label{cAP}
\cA^\cP=\cE^\cP\cup\cF^\cP\cup\cR^\cP\cup\cH^\cP\,.
\end{equation}
Here $\cH^\cP$ denotes \emph{heteroclinic orbits} as defined in \eqref{het}, except that sources and targets $\mathbf{v}_\pm\in\cE^\cP\cup\cF^\cP\cup\cR^\cP$ are now to include frozen and rotating waves.
See \cite{mana97}.

As for hyperbolicity, we can only require \emph{normal hyperbolicity} of $\mathbf{v}\in\cE^\cP\cup\cF^\cP\cup\cR^\cP\subset\mathbf{X}$.
I.e., the trivial eigenvalue and Floquet exponent $\lambda\!=\!0$ of frozen and rotating waves $\mathbf{v}$, respectively, which is generated by the circle action $\bS\bO$, is assumed to be algebraically simple, and any other unrelated purely imaginary Floquet exponents are absent.
In addition to dissipativity, we henceforth impose this generic assumption of normal hyperbolicity on nonlinearities $\mathbf{g}$ in the Sturm classes \eqref{gP}; see \cite{jora10b}.

The main advantage of the $\bS\bO$-equivariant class $\mathbf{g}\!\in\!\SSOP$ is persistence of transversality \eqref{transv} involving the center-unstable and center-stable manifolds of frozen and rotating waves $\mathbf{v}\in\cF^\cP\cup\cR^\cP$.
See \cite{czro08,jora10a}.
The main reasons are normal hyperbolicity and the duplicity of eigenvalues $\lambda$ at homogeneous equilibria $\mathbf{v}\!\in\!\cE^\cP$.
Unless we exclude frozen waves, $\cF^\cP\!=\!\emptyset$, this does not imply $C^0$ Morse-Smale structural stability, however. See \cite{roc12}.
As a remedy, we can again represent the global attractor by a connection graph $\cC^\cP$, where vertices now represent $\bS\bO$ orbits of $\mathbf{v} \in\cE^\cP\cup\cF^\cP \cup\cR^\cP$.
Then the connection graph remains unchanged, up to isomorphism, under small perturbations.
See \cite{firowo04,firowo12b,roc12,roc25,rofiln}.

Global homotopies of nonlinearities $\mathbf{g}$ which preserve normal hyperbolicity and transversality led to a global classification of connection graphs $\cC^\cP$ under periodic boundary conditions.
See \cite{rofiln} for a recent summary.
The successive homotopies of $\mathbf{g}$ led, first, to a freezing of all rotating waves.
Second, for wave speeds frozen at $\mathbf{c}\!=\!0$, the nonlinearity $\mathbf{g}$ was symmetrized by a homotopy from the general $\bS\bO$-equivariant class of $\mathbf{g}\!\in\!\SSOP$ to the $\bO$-equivariant spatially reversible class $\mathbf{g}\!\in\!\SOP$; see \eqref{grev}.
The $\bO$-action is generated by the $\boldsymbol{\xi}$-shifts of $\bSO$ and the \emph{flip}
\begin{equation}
\label{flip}
(\kappa\, \mathbf{u})(\mathbf{t},\mathbf{x}):=\mathbf{u}(\mathbf{t},2-\mathbf{x})) 
\end{equation}
through $\mathbf{x}=0,1$. 
Indeed, the spatially reversible case \eqref{grev} requires the nonlinearity $\mathbf{g}(\mathbf{u},\mathbf{p})$ to be even in $\mathbf{p}\!=\!\mathbf{u_x}$\,.
Zero number dropping \eqref{zdrop} of $z(\mathbf{u}-\kappa \mathbf{u})$ then prohibits rotating waves: $\cRP\!=\!\emptyset$\,.

Spatial reversibility $\mathbf{g}\!\in\!\SOP$ reveals strong relations between Neumann and periodic boundary conditions $\cN$ and $\cP$.
Indeed, let $\mathbf{X}^\kappa$ denote the space of $\kappa$-fixed, spatially 2-periodic $\mathbf{u}\!\in\!\mathbf{X}$, i.e.\ $\kappa\mathbf{u}=\mathbf{u}$ in \eqref{flip}.
Then $\mathbf{X}^\kappa$ is a linear invariant subspace of the $\cP$-semiflow \eqref{gxuux}, \eqref{perbc} on $\mathbf{X}$.
By restriction, in fact, such $\cP$-solutions in $\mathbf{X}^\kappa$, defined for $\mathbf{x}\!\in\!\bS^1\!=\!\bR/2\bZ$, define $\cN$-solutions on $\mathbf{x}\!\in\![0,1]$.
On the other hand, any solution $\mathbf{u}\!\in\!\cAgN$ can be lifted, by reflection through the Neumann boundary, to become a solution
\begin{equation}
\label{NtoP}
\mathbf{u}\in\cAgN\cong(\cAgP)^\kappa:=\cAgP\cap\mathbf{X}^\kappa\,.
\end{equation}
This linearly embeds the Neumann attractors $\cAgN$ into the attractors $\cAgP$ under periodic boundary conditions.

Conversely, spatial reversibility of $\mathbf{g}$ also allows us to shift any spatially 2-periodic frozen wave $\mathbf{v}\!\in\!\cFP$ to be \emph{represented} by Neumann boundaries \eqref{Nbc} on $\mathbf{x}\!\in\![0,1]$, in the connection graph.
From ODE \eqref{gvp}, it now follows that each individual frozen wave orbit $\bS\bO\mathbf{v}\in\cFP$ possesses two distinct such Neumann representatives: $\mathbf{\underline{v}}$ and $\mathbf{\overline{v}}$, as in \eqref{vpair}.
From that point of view, only periodic boundary conditions $\cP$ reveal the hidden $\bO$-symmetry \cite{gomasc84} which leads to the ominous pairing $\mathbf{\underline{v}},\mathbf{\overline{v}}$ in the Neumann problem $\cN$.
For spatially reversible nonlinearities $\mathbf{g}\!\in\!\SOP$, we can therefore proceed via ODE period maps \eqref{Tmap}, exactly as in the Neumann case of section of section \ref{Intro}, to determine all remaining equilibria $\mathbf{v}$.
In conclusion, combining corollary 2.3 of \cite{rofiln} and our present theorem \ref{thmHamrev} validates the following result.
\begin{theorem}\label{thmCPN}
For any nonlinearity $\mathbf{f}\!\in\!\SSOP$ there exists a Hamiltonian (and hence reversible) nonlinearity $\mathbf{g}\!\in\!\SH$ such that
\begin{equation}
\label{CPN}
\cC_\mathbf{f}^\cP\cong\cCP\cong\cCgN /\!\sim\,.
\end{equation}
Here the equivalence relation $\sim$ identifies each vertex pair $\mathbf{\underline{v}},\mathbf{\overline{v}}$, in the Neumann problem of \eqref{vpair}, to a separate single vertex.
\end{theorem}
See figure \ref{figCI3}(f) for an illustration of the vertex identification $\sim$, in the case of the three-dimensional Chafee-Infante problem $\mathrm{CI}_3$\,.
For the detailed fate of the Morse indices $i(\mathbf{v})$ under the transition between $\cP$ and $\cN$, alias state spaces $\mathbf{X}$ to $\mathbf{X}^\kappa$, we refer to the original literature \cite{firowo04,firowo12b,roc12,roc25}. 
See also \cite{rofiln}, which includes an explicit discussion of the 21 connection graphs $\cCgP$ with $\mathbf{g}\!\in\!\SHN\!=\!\SHP$ and $2r+q\leq6$.
As before, $r$ counts the ODE-centers and $q$ the 2-periodic orbits of ODE \eqref{gvp} with some minimal period $T\!=\!2/\ell$.

In section \ref{Sign}, we have already seen how the period maps $T_\mathbf{g}$ of spatially reversible $\mathbf{g}\!\in\!\SON$ lead to Hamiltonian Sturm permutations $\sigma_\mathbf{g}$ which, on the one hand, are in 1-1 correspondence to full lap signatures.
Background section \ref{Back}, on the other hand, has shown how Sturm permutations, in general, and Hamiltonian Sturm permutations, in particular, determine their Neumann connection graphs $\cCgN$. 

The counting results on full lap signatures in our present paper have neither settled the counting of directed Neumann connection graphs $\cCgN$, up to isomorphism, nor the counting of Neumann global attractors $\cAgN$, up to $C^0$ orbit equivalence -- not even for spatially reversible or Hamiltonian nonlinearities $\mathbf{g}\!\in\!\SON\!\supset\!\SHN$.
At best, our results constitute a first step in those directions.
For example, we have already mentioned that our results collaterally address the refined counting of signed Thom-Smale complexes of Neumann Sturm global attractors \cite{firo22}, in the reversible case.
Similarly, the reduction \eqref{CPN} of Neumann connection graphs $\cCgN$ to connection graphs $\cCP$ and the attractor embedding \eqref{NtoP} may serve as an encouragement towards the counting challenges for nonlinearities $\mathbf{g}\!\in\!\SSOP$ under periodic boundary conditions, as well.
And the combinatorics for general $\mathbf{x}$-dependent nonlinearities $\mathbf{g}\!\in\!\SxN$ remain as untouched as ever.


\end{document}